\documentclass[11pt,a4paper,twoside]{article}
\usepackage{graphics}
\usepackage{amssymb}
\usepackage{amsmath}
\usepackage{mathrsfs}
\usepackage{makeidx}
\usepackage{color}
\usepackage{graphicx}
\usepackage{subfigure}
\usepackage{multicol}
\usepackage{multirow}
\usepackage{float}
\usepackage{booktabs}
\usepackage{epsfig}
\usepackage{epstopdf}
\usepackage[colorlinks,
            linkcolor=red,
            anchorcolor=blue,
            citecolor=blue
            ]{hyperref} 
\usepackage{caption}
\usepackage{algorithm}
\usepackage{algorithmic}
\usepackage{latexsym, cite, bm, amsthm, amsmath}
\usepackage{enumerate}
\usepackage{marginnote}

\def\vextra{\vphantom{\vrule height0.4cm width0.9pt depth0.1cm}}
\def \[{\begin{equation}}
\def \]{\end{equation}}
\def\R{\mathbb{R}}
\newtheorem{thm}{Theorem}[section]

\newtheorem{lem}{Lemma}[section]

\newtheorem{alg}{Algorithm}[section]
\newtheorem{rem}{Remark}[section]
\newtheorem{assu}{Assumption}[section]
\newtheorem{exam}{Example}[section]
\numberwithin{equation}{section}

\title{\bf An inexact Douglas-Rachford splitting method for solving absolute value equations}

\usepackage{authblk}
\author{Cairong Chen\thanks{School of Mathematical Sciences, Beihang University, Beijing, 100191, P.R. China. This author is partially supported  by the NSFC grant 11901024 and the China Postdoctoral Science Foundation (2019M660385). Email: \texttt{cairongchen@buaa.edu.cn.}}
\hskip 0.5cm Dongmei Yu\thanks{Institute for Optimization and Decision Analytics, Liaoning Technical University, Fuxin, 123000, P.R. China. This author is supported by the China Postdoctoral Science Foundation (2019M650449). Email: {\tt yudongmei1113@163.com}}
\hskip 0.5cm Deren Han\thanks{Corresponding author. School of Mathematical Sciences, Beijing Advanced Innovation Center for Big Data and Brain Computing (BDBC), Beihang University, Beijing, 100191, P.R. China. This author is partially supported by the NSFC grants 11625105, 11926358. Email: \texttt{handr@buaa.edu.cn.}}
}

\begin{document}
\date{\today}
\maketitle


\begin{quote}
{\bf Abstract:} The last two decades witnessed the increasing of the interests on the absolute value equations (AVE) of finding $x\in\mathbb{R}^n$ such that  $Ax-|x|-b=0$, where $A\in \mathbb{R}^{n\times n}$ and $b\in \mathbb{R}^n$.
In this paper, we pay our attention on designing efficient algorithms. To this end, we reformulate  AVE to a generalized linear complementarity problem (GLCP), which, among the equivalent forms, is the most economical one in the sense that it does not increase the dimension of the variables. For solving the GLCP, we propose  an inexact Douglas-Rachford splitting method which can adopt a relative error tolerance. As a consequence, in the inner iteration processes, we can employ the LSQR method ([C.C. Paige and M.A. Saunders, \textit{ACM Trans. Mathe. Softw. (TOMS)}, 8 (1982), pp. 43--71]) to find a qualified approximate solution for each subproblem, which makes the cost per iteration very low. We prove the convergence of the algorithm and establish its global linear rate of convergence.
Comparing results with the popular algorithms such as the exact generalized Newton method [O.L. Mangasarian, \textit{Optim. Lett.}, 1 (2007), pp. 3--8], the inexact semi-smooth Newton method [J.Y.B. Cruz, O.P. Ferreira and L.F. Prudente, \textit{Comput. Optim. Appl.}, 65 (2016), pp. 93--108] and the exact SOR-like method [Y.-F. Ke and C.-F. Ma, \textit{Appl. Math.
Comput.}, 311 (2017), pp. 195--202] are reported, which indicate that the proposed algorithm is very promising. Moreover, our method also extends the range of numerically solvable of the AVE; that is, it can deal with not only the case that $\|A^{-1}\|<1$, the commonly used in those existing literature, but also the case where $\|A^{-1}\|=1$.

{\small

}

\end{quote}
\newpage
\section{Introduction}\label{sec:intro}
 The system of absolute value equations (AVE) is of finding a vector $x\in \mathbb{R}^n$ such that
\begin{equation}\label{eq:ave}
    Ax-|x|-b=0,
\end{equation}
where $A\in \mathbb{R}^{n\times n}, b\in \mathbb{R}^n$, and $|x|$ denotes the component-wise absolute value of the vector $x \in \mathbb{R}^n$ whose $i$-th value is $x_i$ if $x_i\ge 0$ and $-x_i$ otherwise. The AVE \eqref{eq:ave} is a special case of the more general AVE
\begin{equation}\label{eq:gave}
 Ax +B|x|=b
\end{equation}
with $A, B\in \mathbb{R}^{m \times n}$, which is introduced in \cite{rohn2004} and further investigated in \cite{mang2007,mang2009,prok2009} and references therein. On the other hand, if $B$ is nonsingular,
then~\eqref{eq:gave} can be reduced to the form of \eqref{eq:ave}. For simplicity, we focus our attention on
\eqref{eq:ave}.

Due to the combinatorial character introduced by the absolute operator, determining the existence of a solution to the AVE is NP-hard  \cite{mang2007}.  Moreover, if AVE is solvable, checking whether the AVE has a unique solution or multiple solutions is NP-complete \cite{prok2009}. A remarkable and commonly used result on the solvability of the AVE \eqref{eq:ave} is established in \cite{mame2006}, which stated  that if $\|A^{-1}\|<1$ (or equivalently, $\sigma_{\min}(A)>1$), then the AVE~\eqref{eq:ave} has a unique solution $x^*$ for any $b \in \mathbb{R}^{n}$; here and throughout the paper, $\|\cdot\|$ of a matrix denotes the $2$-norm of the matrix  and $\sigma_{\min}(A)$ denotes the smallest singular value of~$A$.

Under some assumption, the  AVEs \eqref{eq:ave} and \eqref{eq:gave} are equivalent to the linear complementarity problem (LCP)  of finding a vector $u\in \mathbb{R}^l$ such that
\begin{equation}\label{eq:lcp}
u\ge 0, \quad Mu+q\ge 0 \quad \hbox{and}\quad \left\langle u,Mu+q \right\rangle=0.
\end{equation}
One of such reformulations is from \cite[Proposition~3.1]{prok2009}, where the AVE~\eqref{eq:ave} is reformulated to~\eqref{eq:lcp} with
\[\label{prok}
M=\left[\begin{array}{ccc}
-I & 2I & 0\\
A & -(I+A) & 0\\
-A & A+I &0
\end{array}\right].
\]
Note that here $l=3n$, i.e., the dimension of the variable in {\color{blue} LCP~\eqref{eq:lcp} with~\eqref{prok} } is three times that in the original problem
\eqref{eq:ave}. Then, \cite{mame2006} and \cite{mang2007} propose new LCP reformulations:
\[\label{lcps}
u\ge 0,\quad (A+I)(A-I)^{-1}u+q\ge0 \quad \hbox{and}\quad \left\langle u,(A+I)(A-I)^{-1}u+q\right\rangle=0.
\]
In~\cite{mame2006},
 $$q=\left[(A+I)(A-I)^{-1}-I\right]b, \quad u=(A-I)x-b;$$
 and  in~\cite{mang2007},
$$ q=(I-A)^{-1}b, \quad u=\frac{1}{2}((A-I)x-b).$$
While the reformulation \eqref{lcps} does not increase the dimension of the variable, it limits to the case where
$1$ is not an eigenvalue of $A$.
Recently, another complementarity reformulation is proposed in \cite{abhm2018}:
\begin{equation*}
\left\{\begin{array}{l}A(x^+-x^-)-(x^++x^-)=b\\
x^+\ge 0, x^-\ge 0\\
\langle x^+,x^-\rangle=0.\end{array}
\right.
\end{equation*}

LCP subsumes many mathematical programming problems and such equivalent reformulations attract much attention from optimization area. Mangasarian \cite{mang2007, mang2007a} suggest to solve AVEs \eqref{eq:ave} and \eqref{eq:gave} via concave minimization and propose to solve it by a finite succession of linear programs. The concave minimization problem
of \eqref{eq:gave} is
\begin{equation}\label{concave}
\begin{array}{crll}
\displaystyle\min_{(x,t,s)\in\R^{n+n+m}}&\epsilon(-\langle e,|x|\rangle+\langle e,t\rangle)+\langle e,s\rangle&&\\
\hbox{such\;that}&-s  \le Ax+Bt-b \le s,&&\\
&-t \le x \le t.&&\end{array}
\end{equation}
Though there is no theoretical guarantee, it found that the algorithms terminate at a local minimum that solves the  AVE in almost all solvable random problems tested in the paper. Utilizing the equivalent  reformulation \eqref{lcps} for the case where $1$ is not an eigenvalue of $A$, \cite{maer2018} studies AVE and proposes a dynamical model, whose stability and global convergence were proved under suitable conditions. For the special case of \eqref{eq:ave} where $A$ is symmetric and positive definite, Noor \cite{noin2012} reformulates it as an optimization problem and proposes a Gauss-Seidel method.

The AVEs \eqref{eq:ave} and \eqref{eq:gave} can be  viewed as  special cases of the nonlinear equation
\begin{equation*}
F(x)=0,
\end{equation*}
where $F$ a nonsmooth mapping from $\R^n$ to $\R^n$ for  \eqref{eq:ave} and from $\R^m$ to $\R^n$ for \eqref{eq:gave}, respectively. From this point of view,  numerical algorithms for solving nonsmooth equations can be adopted to
 find a solution of the AVEs. In \cite{mang2009}, a generalized Newton method is directly utilized to solve \eqref{eq:ave} and the iteration scheme is
 \[\label{en}
  \left[A-D(x^k)\right]x = b,\]
 where $D(x) \doteq \textbf{diag}(\textbf{sign}(x))$ with $\textbf{sign}(x)$ denoting a vector with components equal to $-1,\, 0$ or $1$, respectively, depending on whether the corresponding component of the vector $x$ is negative, zero or positive; and for $x \in \mathbb{R}^{n}$, $\textbf{diag}(x)$ representing a diagonal matrix with $x_i$ as its diagonal entries for every $i=1,2,\cdots,n$.
The reported numerical results are much promising. Under the conditions that if $\|A^{-1}\|<1/3$ and $\|D(x^k)\|\ne 0$, the exact semi-smooth Newton method is well defined and linearly converges  to the unique solution $x^*$ of the AVE~\eqref{eq:ave}.  Then, a smoothing Newton method is proposed in \cite{caqz2011}, and under suitable conditions on the smoothing function, it proves that the method is convergent and the rate of convergence can be quadratic.

 To alleviate the burden of solving system of linear equation during the  Newton iteration~\eqref{en}, the inexact semi-smooth Newton method is proposed~\cite{crfp2016}. Instead of solving the linear equation exactly as in Newton's method \eqref{en}, it only requires to find $x \in\R^n$ such that
 \begin{equation*}
               \left[A-D(x^k)\right]x = b+r_k,\quad \text{where} \quad \left\|  r_k\right\|
               \le \theta \left\|Ax^k-|x^k|-b\right\|.
            \end{equation*}
As stated in~\cite{crfp2016}, if $\|A^{-1}\|<1/3$ and
 \begin{equation}\label{eq:theta}
               0\le \theta < \frac{1-3\|A^{-1}\|}{\|A^{-1}\|(\|A\|+3)},
 \end{equation}
then the inexact semi-smooth Newton method is well defined and globally converges to the unique solution $x^*$ of the AVE~\eqref{eq:ave}, and its rate of convergence is $Q$-linearly.

Gauss-Seidel and SOR methods are classical numerical algorithms for solving linear equations \cite{saad2003}. A generalized Gauss-Seidel iteration method is proposed in \cite{edhs2017}, and sufficient conditions guaranteeing global convergence are also presented.
Reformulating AVE \eqref{eq:ave} as
\begin{equation*}
\left\{\begin{array}{l}
Ax-y=b,\\
y-|x|=0,
\end{array}\right.
\end{equation*}
and then define
\begin{align*}
\boldsymbol{Az}\doteq
\left(\begin{array}{cc}
A    &   -I \\
-D(x) &   I
\end{array}\right)
\left(\begin{array}{c}
x    \\   y
\end{array}\right)
=
\left(\begin{array}{c}
b    \\  0
\end{array}\right)
\doteq\boldsymbol b.
\end{align*}
By splitting the coefficient matrix $\boldsymbol{A}= \mathcal{D} - \mathcal{L} - \mathcal{U} $ with
\begin{align*}
\mathcal{D}=
\left(\begin{array}{cc}
 A   &   0 \\
 0   &   I
\end{array}\right),\quad
\mathcal{L}=
\left(\begin{array}{cc}
0     &   0 \\
D(x)   &   0
\end{array}\right),\quad
\mathcal{U}=
\left(\begin{array}{cc}
0      &   I \\
0      &   0
\end{array}\right),
\end{align*}
a SOR-like iteration method  can be derived
\begin{equation}\label{eq:equiv}
\left(\begin{array}{c}
 x^{(k+1)}   \\   y^{(k+1)}
\end{array}\right)
=\mathcal{M}_\omega
\left(\begin{array}{c}
 x^{(k)}   \\   y^{(k)}
\end{array}\right)
+\omega(\mathcal{D} - \omega \mathcal{L})^{-1}
\left(\begin{array}{c}
 b    \\  0
\end{array}\right),
\end{equation}
where $\mathcal{M}_\omega = (\mathcal{D} - \omega \mathcal{L})^{-1}\left[(1-\omega \mathcal{D})+\omega\mathcal{U}\right]$, and $\omega>0$ is the iteration parameter. For judiciously selected parameter $\omega$, \cite{kema2017} proves its convergence and \cite{chyh2020} suggests a scheme on choosing the optimal parameter.


Note that the widely used sufficient condition $\|A^{-1}\|<1$, ensuring that the AVE \eqref{eq:ave} possesses a unique solution, is derived from the equivalent reformulation of \eqref{eq:ave} as a LCP~\eqref{lcps} \cite{mame2006}. As a consequence, algorithms solving \eqref{eq:ave} via solving \eqref{lcps} inevitably involve computation of $(A-I)^{-1}$. In this paper, we first give a new look on the condition, where we derive a new sufficient condition ensuring the uniquely solvable of the AVE~\eqref{eq:ave}. While the new condition is stronger than the popular one, i.e., $\|A^{-1}\|<1$,  it provides some useful insights of the problem.

Our main concern is to  develop numerical algorithms for solving the AVE~\eqref{eq:ave}, which has two main features. That is, for the case where $ \|A^{-1}\|<1$, the new algorithm is more efficient. More importantly, it  can deal with the case that  $\|A^{-1}\|= 1$. Hence, we need consider two possible cases, i.e., the solution set of the AVE \eqref{eq:ave} is nonempty and it does not possess a solution. For the first case, we prove that the new algorithm converges in a global linear rate to a solution. And for the second case that the solution set is empty, we prove that the generated sequence diverges to infinity, providing a numerical checking method for the existence of a solution.


For the general case considered in this paper, i.e., $ \|A^{-1}\|\le 1$, we can not convert the  AVE~\eqref{eq:ave}  to the LCP~\eqref{lcps}, due to the possible singularity of $A-I$. On the other hand, { if we convert the  AVE~\eqref{eq:ave} to
\eqref{eq:lcp} with~\eqref{prok} or \eqref{concave}, the scale will increase}.
Here, we are interested in its another reformulation of finding $x\in \mathbb{R}^n$ such that
\begin{equation}\label{eq:glcp}
    Q(x)\doteq Ax+x-b\ge 0,\quad  F(x)\doteq Ax-x-b\ge 0 \quad \hbox{and}\quad \left\langle  Q(x), F(x)\right\rangle=0.
\end{equation}
The problem \eqref{eq:glcp} is a generalized linear complementarity problem (GLCP) \cite{fapa2007}, and
the equivalence between the AVE~\eqref{eq:ave} and the GLCP~\eqref{eq:glcp} is proved in \cite{mame2006}.
GLCP~\eqref{eq:glcp} is a special case of the general variational inequalities (GVI) problem of
\begin{equation}\label{eq:gvi}
\text{find an}~x^*,~ \text{such that}~Q(x^*)\in \Omega, ~ \left\langle v-Q(x^*),F(x^*)\right\rangle\ge 0,\; \forall\, v\in \Omega,
\end{equation}
with $\Omega=\{x\in\mathbb{R}^n|x\ge 0\}$. When $Q(x)\equiv x$, the GVI problem~\eqref{eq:gvi} reduces to the classical VI problem
\begin{equation}\label{eq:vi}
\text{find an}~x^*\in \Omega,~ \left\langle v-x^*,F(x^*)\right\rangle\ge 0, \; \forall \, v\in \Omega.
\end{equation}
The VI problem~\eqref{eq:vi} has important applications in many fields including economics, operations research and nonlinear analysis and a survey of methods for solving it can be found in \cite{hapa1990,cops1992,fapa2007}. 

Our new algorithm is inspired by the operator splitting methods for solving the generalized equation
\begin{equation*}
0\in T(x),\end{equation*}
where $T:\R^n\rightrightarrows\R^n$ is a set-valued mapping and has the form $T:=\Theta+\Phi$, and $\Theta+\Phi$ are two maximal monotone mappings. We focus on the recently resurrected Douglas-Rachford splitting method \cite{dora1956}, whose recursion is very simple: at the $k$th {iteration}, the next {iteration} is generated via solving the linear equation
\[\label{h1}
2Ax-2Ax^k+\gamma\rho(x^k)G^{-1}(Ax^k-|x^k|-b)=0,\]
where $G$ is a symmetric positive definite matrix taking the role of preconditioner, and $\rho(x^k)$ is a parameter which can be evaluated easily. To reduce the computational cost per iteration, we allow to solve the equation \eqref{h1} approximately and the accuracy criterion is constructive. We prove that, if the solution set of the AVE \eqref{eq:ave}, denoted by $S^*$, is nonempty, then the algorithm converges globally to a solution, and the rate of convergence is linear; if $S^*=\emptyset$, then the generated sequence diverges to infinity; hence our algorithm naturally possesses the ability of numerically checking the existence of solutions of AVE \eqref{eq:ave} (not necessary unique).


The rest of this paper is organized as follows. In Section~\ref{sec:pre} we present some notations, classical definitions and summarize some preliminary results, relevant to our later developments. Section~\ref{sec:newlook} gives a new sufficient condition and a new look on the sufficient condition $\sigma_{\min}(A)>1$ for the AVE~\eqref{eq:ave} being uniquely solvable for any $b\in\mathbb{R}^n$.  Section~\ref{sec:met} describes the methods and shows their contractive properties and proves the convergence of the methods. In Section~\ref{sec:numer}, three numerical examples are given to demonstrate our claims. Finally, some concluding remarks are given in Section~\ref{sec:con}.

\section{Preliminaries}\label{sec:pre}
In this section, we summarize some  notations, classical definitions and  auxiliary results which will be used  throughout the paper.

{We use $\mathbb{R}^{n\times n}$ ($\mathbb{C}^{n\times n}$) to denote the set of all $n \times n$ real (complex) matrices and $\mathbb{R}^{n}= \mathbb{R}^{n\times 1}$ ($\mathbb{C}^{n}= \mathbb{C}^{n\times 1}$). $I$ is the identity matrix with suitable dimension. $| \cdot |$ denotes absolute value for real scalar or modulus for complex scalar. We denote by $\lambda_{\max}(A)$ and $\lambda_{\min}(A)$ the largest eigenvalue and the smallest eigenvalue of $A$, respectively. The transposition of a matrix or vector is denoted by $\cdot ^T$. The inner product of two vectors in $\mathbb{R}^n$ is defined as $\langle x, y\rangle\doteq x^Ty= \sum\limits_{i=1}^n x_i y_i$ and $\| x \|\doteq\sqrt{\langle x, x\rangle} $ denotes the $2$-norm of vector $x\in \mathbb{R}^{n}$.
$\textbf{tridiag}(a, b, c)$ denotes a matrix that has $a, b, c$ as the subdiagonal, main diagonal and superdiagonal entries in the matrix, respectively. The projection mapping from $\mathbb{R}^n$ onto $\Omega$, denoted by $P_{\Omega}$, is defined as $P_{\Omega}[x]=\arg\min\{\|x-y\|:y\in \Omega\}$. $\emptyset$ denotes the empty set.

A complex scalar $\lambda$ is called an eigenvalue of the square matrix $A$ if a nonzero vector $u\in \mathbb{C}^n$ exists such that $Au=\lambda u$. The set of all the eigenvalues of $A$ is called the spectrum of $A$ and is denoted by $\lambda(A)$. The square roots of the eigenvalues of $A^TA$ are called singular values of a real matrix $A$, and is denoted by $\sigma(A)$. The maximum modulus of the eigenvalues of $A$ is called the spectral radius and is denoted by $\rho(A)\doteq \max\limits_{\lambda\in \lambda(A)}|\lambda|$. It is easy to see that $\rho(A)\le \|A\|$ for an arbitrary square matrix $A$; see for instance \cite[Page~22]{saad2003}. Matrix $A$ is symmetric if $A^T=A$. A matrix $A\in \mathbb{R}^{n\times n}$ is said to be positive definite or positive real if $\langle Au, u\rangle>0$ for all $0\ne u\in \mathbb{R}^n$; see for example \cite[Page~29]{saad2003}; and a symmetric positive real matrix is said to be symmetric positive definite. $A\in \mathbb{R}^{n\times n}$ is positive real if and only if the symmetric part of $A$ is symmetric positive definite; in addition, if  $A\in \mathbb{R}^{n\times n}$ is positive real  then any eigenvalue $\lambda$ of $A$ satisfies $\textit{Re}(\lambda)>0$, where $\textit{Re}(\lambda)$ denotes the real part of $\lambda$; see for example \cite[Pages~30-31]{saad2003}. A matrix $A\in \mathbb{R}^{n\times n}$ is said to be a $P$-matrix if all its principal minors are positive~\cite[Page~147]{cops1992}. It is well known that a symmetric matrix is positive definite if and only if it is a $P$-matrix.
}

Let $F$ and $Q$ be continuous mappings from $\R^n$ to $\R^n$. A generalized complementarity problem, denoted by GCP($F,Q$) is finding $x\in\R^n$ such that
\[\label{gcp}
0\le F(x)\bot Q(x)\ge 0,\]
where for two vectors $u,v$ with the same dimension, the symbol `$u\bot v$' denotes $\langle u,v\rangle=0$.
If both $F$ and $Q$ are linear mappings, we call \eqref{gcp} a generalized linear complementarity problem and is denoted by  GLCP($F,Q$).  The GCP \eqref{gcp} is monotone if
  $$[Q(u)-Q(v)]^T [F(u)-F(v)]\ge 0,  \;  \forall \,u, v \in \mathbb{R}^n;$$
and it is strongly monotone if there is a $\mu>0$ such that
 $$[Q(u)-Q(v)]^T [F(u)-F(v)]\ge \mu \|u-v\|^2,  \; \forall \,u, v \in \mathbb{R}^n.$$
The operator $Q + F$ defined as in \eqref{gcp} is nonsingular (as defined by He \cite{he1999}) if there is a $\nu>0$ such that
  $$ \left\|(Q+F)(u-v)\right\|\ge \nu \|u-v\|, \;  \forall \,u, v \in \mathbb{R}^n.$$
Recall that for the GLCP($F,Q$) \eqref{eq:glcp} arises from the AVE \eqref{eq:ave}, we have
$$ F(x)\doteq Ax-x-b\qquad\hbox{and}\qquad { Q(x)\doteq Ax+x-b}.$$
Consequently,
  \begin{align*}
    [Q(u)-Q(v)]^T [F(u)-F(v)]
        & =[ A(u-v)+(u-v)]^T [A(u-v)-(u-v)] \\
        &=(u-v)^T A^T A(u-v)-(u-v)^T A^T(u-v)\\
           &\quad + (u-v)^T A (u-v) -(u-v)^T (u-v)\\
            &=(u-v)^T(A^T A-I)(u-v),
\end{align*}
implying that the GLCP \eqref{eq:glcp} is monotone if and only if $ \sigma_{\min} (A)\ge 1$ or equivalently $\|A^{-1}\|\le 1$; and  it is strongly monotone if and only if $ \sigma_{\min} (A)> 1$ or equivalently $\|A^{-1}\|< 1$.
In addition,
\begin{align*}
\|(Q+F)(u-v)\|^2&= \|(Q+F)(u)-(Q+F)(v)\|^2\\
&= \|2Au-2b-(2Av-2b)\|^2\\
&= 4(u-v)^TA^TA(u-v)\\
&\ge 4\sigma_{\min}^2(A)\|u-v\|^2,
\end{align*}
and the operator $Q + F$ defined as in \eqref{eq:glcp} is nonsingular if and only if $A$ is nonsingular.


{ The GCP($F,Q$)} is a special case of the  generalized variational inequality problem GVI($F,Q;\Omega$) of finding a vector $x^*\in\R^n$ such that
\[\label{gvi}
~Q(x^*)\in \Omega, ~ \left\langle v-Q(x^*),F(x^*)\right\rangle\ge 0,\; \forall\, v\in \Omega,\]
where $\Omega$ is a nonempty closed subset of $\R^n$; and GCP($F,Q$) corresponds to the special case where  $\Omega=\{x\in \mathbb{R}^n|x\ge0\}$. For a nonempty closed subset $\Omega$ of $\R^n$, a fundamental property of the projection from $\R^n$ onto $\Omega$ is that
\[\label{pp}
\langle u-w, v-w\rangle \le 0, \; \forall u\in\R^n, \;\forall\, v, w\in\Omega,\]
holds if and only if $w=P_\Omega[u]$. Using this property, we can find that GVI($F,Q;\Omega$) \eqref{gvi}
 is equivalent to the following nonsmooth projection equation
\begin{equation}\label{eq:pe}
Q(x^*) = P_{\Omega}\left[Q(x^*)-F(x^*)\right].
\end{equation}
%
Let
\begin{equation}\label{eq:res}
e(x)\doteq Q(x) - P_{\Omega}\left[Q(x)-F(x)\right]
\end{equation}
be the residual of the equation \eqref{eq:pe}.
%
Now we are in position to prove the following basic theorem.

\begin{thm}\label{thm:aveequivpe}
$x^*$ is a solution of the AVE~\eqref{eq:ave} if and only if $e(x^*)=0$. Furthermore, we have
\[\label{ex}e(x)=Ax-|x|-b.\]
\end{thm}

\begin{proof}
The first part of this theorem can be similarly proved in the light of that of Theorem~1 in \cite{he1999}. In fact, it follows from \eqref{eq:pe} and the property~\eqref{pp} of the projection operator that
$$
\begin{array}{rcl}e(x^*)=0&\Longleftrightarrow&  { Q(x^*) = P_{\Omega}\left[Q(x^*)-F(x^*)\right]}\\
&\Longleftrightarrow& \langle v-Q(x^*),Q(x^*)-(Q(x^*)-F(x^*))\rangle\ge 0,\;\forall \, v\in\Omega\\
&\Longleftrightarrow& \langle v-Q(x^*),F(x^*)\rangle\ge 0,\;\forall\, v\in\Omega.
\end{array}
$$
For the second assertion, we have that for $\Omega=\{x\in \mathbb{R}^n|x\ge0\}$,
$$
P_{\Omega}[x]=\max\{x,0\}=(|x|+x)/2.$$
Hence, it follows from the definitions of $Q$ and $F$ in \eqref{eq:glcp} that
\begin{align*}
e(x)&= Q(x) - P_{\Omega}\left[Q(x)-F(x)\right]\\
&=Ax + x -b-P_{\Omega}\left[Ax + x -b-(Ax - x -b)\right]\\
&=Ax + x -b-P_{\Omega}\left[2 x \right]\\
&=Ax+x-b-(|x|+x)\\
&=Ax-|x|-b,
\end{align*}
from which we complete the proof.
\end{proof}


We will end this section by proving the following theorem. The proof is inspired by that of Theorem~2 in~\cite{he1999}.

\begin{thm}\label{thm:contrac}
If $x^*$ is a solution of the AVE \eqref{eq:ave}, then
\begin{equation}\label{ie:contr}
{ \left\langle A(x-x^*), e(x) \right\rangle } \ge\frac{1}{2}\left\{ \left\|e(x)\right\|^2+\langle x-x^*, (A^TA-I)(x-x^*)\rangle\right\},\; \forall \,x\in \mathbb{R}^n.
\end{equation}
\end{thm}
\begin{proof} Recall that the AVE \eqref{eq:ave} is equivalent to the GLCP~\eqref{eq:glcp}; hence, $x^*$ is also a solution of GLCP($Q,F$), implying $Q(x^*)\in\Omega$ and
$$
\langle v-Q(x^*),F(x^*)\rangle \ge 0,\;\forall\, v\in\Omega,$$
where $\Omega=\{x\in \mathbb{R}^n|x\ge 0\} \subseteq \mathbb{R}^n$.
Let $x\in\R^n$ be an arbitrary vector and setting $v:= P_\Omega[Q(x)-F(x)]$ leads to
\begin{align}\label{equ7}
\langle P_\Omega[Q(x)-F(x)]-Q(x^*), F(x^*)\rangle\ge 0, \;\forall\, x\in\R^n.
\end{align}
On the other hand, setting $u:=Q(x)-F(x)$ and $v:=Q(x^*)$ we have
\begin{align}\label{equ6}
\langle Q(x)-F(x)-P_\Omega[Q(x)-F(x)], P_\Omega[Q(x)-F(x)]-Q(x^*)\rangle \ge 0, \;\forall\, x\in\R^n.
\end{align}
Adding \eqref{equ6} and \eqref{equ7}, and using \eqref{eq:res} we obtain
{ $$
\langle Q(x)-Q(x^*)-e(x), e(x)-F(x)+F(x^*)\rangle\ge 0, \;\forall\, x\in\R^n.$$ }
Rearranging terms, we have
\begin{align}\nonumber
&\left\langle \left(Q(x)-Q(x^*)\right)+\left(F(x)-F(x^*)\right), e(x)\right\rangle\\\nonumber
&\qquad \ge \left\|e(x)\right\|^2+\langle Q(x)-Q(x^*),F(x)-F(x^*)\rangle.\nonumber
\end{align}
Noting that for $Q$ and $F$ defined in \eqref{eq:glcp}, we have
\[\label{QF}
Q(x)+F(x)=2(Ax-b)\]
and
$$
\langle Q(x)-Q(x^*), F(x)-F(x^*)\rangle=\langle x-x^*,(A^TA-I)(x-x^*)\rangle,$$
the assertion then  follows immediately.
\end{proof}


Our proof of the convergence of the proposed algorithms is based on  the {\it Fej$\acute{{e}}$r monotonicity} \cite[Def. 5.1]{bauschke2011convex} of the iterates with respect to the solution set. The  {Fej\'er monotonicity} of a sequence {$\{v^k\}$} with respect to a closed convex set $\Omega$ says that for two consecutive point $v^{k+1}$ and $v^k$, we have
{
\begin{equation*}
\hbox{dist}(v^{k+1},\Omega)\le \hbox{dist}(v^{k},\Omega)
,\end{equation*}
}
where
$$
\hbox{dist}(v,\Omega):=\inf_{u\in\Omega} \|u-v\|$$
denotes the distance between a point $v$ and $\Omega$. 
We now summarize a Fej\'{e}r monotonicity technique (see, for instance, \cite[Lemmas 1.2]{bot2019proximal}, which is useful in our  convergence analysis.

\begin{lem}\label{fejer}Let $\{\xi_n\}$ be a sequence of real numbers and $\{\omega_n\}$ be a sequence of real nonnegative numbers. Assume that $\{\xi_n\}$ is bounded from below and that, for any $n\ge 0$,
$$
\xi_{n+1}\le \xi_n-\omega_n.$$
Then the following statements hold:
\begin{enumerate}[(i)]
\item the sequence $\{\omega_n\}$ is summable, i.e., $\sum_{n=1}^\infty \omega_n<\infty$;
\item the sequence $\{\xi_n\}$ is monotonically decreasing and convergent. \qed
\end{enumerate}
\end{lem}

\section{A new look on the sufficient condition $\sigma_{\min}(A)>1$}\label{sec:newlook}
It was proved in \cite{mame2006} that the AVE~\eqref{eq:ave} is uniquely solvable for any $b\in \mathbb{R}^n$ if $\sigma_{\min}(A)>1$, or equivalently $\|A^{-1}\|<1$. The proof there was finished by first reducing the AVE~\eqref{eq:ave} to the LCP~\eqref{eq:lcp} with $M=(A+I)(A-I)^{-1}$ and then using 3.3.7 in \cite[Page~148]{cops1992}; see~\cite[Proposition~3]{mame2006} for detail. In this section, we will introduce a new sufficient condition for the AVE~\eqref{eq:ave} being uniquely solvable for any $b\in \mathbb{R}^n$ and then describe a new look on the sufficient condition $\sigma_{\min}(A)>1$ from the view of Banach fixed-point theorem {(see for example~\cite[Page~144]{fapa2007})}.

Before that, we  give an example to show that the condition $\sigma_{\min}(A)>1$ is not necessary for the AVE~\eqref{eq:ave} being uniquely solvable for any $b\in \mathbb{R}^n$. The reason is that, in general, a positive definite matrix is not equivalent to a $P$-matrix. For example, let $A=\left[\begin{array}{cc} 1& 2\\ -\frac{2}{3} &1\end{array}\right]$, then we have $\sigma_{\min}(A)=1$ and the AVE~\eqref{eq:ave} is uniquely solvable for any $b\in \mathbb{R}^n$ since $M=(A+I)(A-I)^{-1}=\left[\begin{array}{cc} 1& -3\\ 1 &1\end{array}\right]$ is a $P$-matrix.  However, this $M$ is not positive definite since $u^TMu=0$ for any { $u=[u_1,u_2]^T$ with $u_1=u_2$.} In the following, based on Theorem~\ref{thm:aveequivpe}, we will give another constructed proof for this sufficient condition by constructing a special fixed-point iteration.

Now we are in position to give a new sufficient condition for the AVE~\eqref{eq:ave} being uniquely solvable for any $b\in \mathbb{R}^n$. Let
$$
\mathcal{F}(x)= x-\nu e(x)
$$
with $\nu\neq 0$, then Theorem~\ref{thm:aveequivpe} implies that $x^*$ is a solution of the AVE~\eqref{eq:ave} if and only if $\mathcal{F}(x^*)=x^*$, that is, $x^*$ is a fixed point of $\mathcal{F}$. This provides a way to derive a sufficient condition for the AVE~\eqref{eq:ave} having a unique solution. In fact, {for any $x,\;y\in\mathbb{R}^n$}, since
\begin{align*}
\left\|\mathcal{F}(x)-\mathcal{F}(y)\right\|&=\|(I-\nu A)(x-y)+\nu (|x|-|y|)\|\\
&\le \|I-\nu A\|\cdot \|x-y\|+\nu \|x-y\|\\
&=(\|I-\nu A\|+\nu)\|x-y\|,
\end{align*}
$\mathcal{F}$ is a contraction map if $\|I-\nu A\|+\nu \in (0,1)$; here we have used the inequality $\left\||x|-|y|\right\|\le \|x-y\|$  in the first inequality.
Thus, according to the Banach fixed-point theorem, if there exists $\nu \in (0,1)$ such that $\|I-\nu A\|<1-\nu$, the map $\mathcal{F}$ has a unique fixed point in $\mathbb{R}^n$; that is, the AVE~\eqref{eq:ave} has a unique solution for any $b\in \mathbb{R}^n$. In conclusion, we have the following existence theorem.

\begin{thm}\label{thm:newsc}
The AVE~\eqref{eq:ave} is uniquely solvable for any $b\in \mathbb{R}^n$ if there exists $\nu \in (0,1)$ such that $\|I-\nu A\|<1-\nu$. \qed
\end{thm}

Note that for $\nu \in (0,1)$,
\begin{align*}
\|I-\nu A\|<1-\nu &\Longleftrightarrow \|I-\nu A\|^2 < (1-\nu)^2\\
&\Longleftrightarrow \lambda_{\max}\left((I-\nu A)^T(I-\nu A)\right) < (1-\nu)^2\\
&\Longleftrightarrow \max_{0\ne x\in\mathbb{R}^n}\frac{x^T(I-\nu A)^T(I-\nu A)x}{x^Tx}< (1-\nu)^2\\
&\Longleftrightarrow \left[ x^T(A^TA-I)x\right]\cdot \nu^2 - \left[2x^T(A-I)x\right]\cdot \nu < 0,\;\forall\;0\ne x\in \mathbb{R}^n,
\end{align*}
which is satisfied if $ x^T(A^TA-I)x>0$ and $x^T(A-I)x>0$ for all $0\ne x \in \mathbb{R}^n$ or $ x^T(A^TA-I)x<0$ and $x^T(A-I)x>0$ for all $0\ne x \in \mathbb{R}^n$. However, $ x^T(A^TA-I)x<0$ for all $0\ne x \in \mathbb{R}^n$ implies that $\|A\|<1$, thus $\rho(A)\le \|A\|<1$, which contracts to $x^T(A-I)x>0$ for all $0\ne x \in \mathbb{R}^n$. { In fact, the latter implies that $\textit{Re}(\lambda-1)>0$, that is $\textit{Re}(\lambda)>1$, where $\lambda$ is any eigenvalue of $A$,} which implies $\rho(A)>1$. Therefore, there exists $\nu \in (0,1)$ such that $\|I-\nu A\|<1-\nu$ is guaranteed if $\sigma_{\min}(A)>1$ and $A-I$ is positive real. In particular, if $A$ is symmetric positive definite, the new sufficient condition is equivalent to $\sigma_{\min}(A)>1$. However, in general, the positive real condition is necessary. For example, let {$A=\textbf{diag}([1.8,-2]^T)$}, which satisfies $\sigma_{\min}(A)>1$ but $A-I$ is not positive real, then for all $\nu \in (0,1)$, we have $\|I-\nu A\|>1-\nu$, a contradiction. In conclusion, our sufficient condition is stronger than that of \cite{mame2006}. However, it guarantees the convergence of the fixed-point contraction algorithm $x^{k+1}=\mathcal{F}(x^k),\, k=0,1,2,\cdots$, which is inverse-free. We will not further study the property of this fixed-point iteration since it is not the main subject of this paper.

Finally, we describe a new proof of the sufficient condition $\sigma_{\min}(A)>1$ which guarantees the AVE~\eqref{eq:ave} has a unique solution for any $b\in \mathbb{R}^n$. Since $A$ is nonsingular, we can conclude that $x^*$ is a solution of the AVE~\eqref{eq:ave} if and only if $\mathcal{\tilde{F}}(x^*)=x^*$ with \begin{equation}\label{eq:fix}
\mathcal{\tilde{F}}(x)=x-\nu A^{-1}e(x), \quad \nu \in (0,1).
\end{equation}
Similarly, we have
\begin{align*}
\left\|\mathcal{\tilde{F}}(x)-\mathcal{\tilde{F}}(y)\right\|&=\|(1-\nu )(x-y)+\nu A^{-1} (|x|-|y|)\|\\
&\le (1-\nu)\cdot \|x-y\|+\nu \|A^{-1}\| \|x-y\|\\
&=(1-\nu+\nu\|A^{-1}\| )\|x-y\|,
\end{align*}
from which we can conclude that if $\|A^{-1}\|<1$, $\mathcal{\tilde{F}}$ is contractive and according to the Banach fixed-point theorem, $\mathcal{\tilde{F}}(x)=x$ has a unique solution in $\mathbb{R}^n$, that is, the AVE~\eqref{eq:ave} has a unique solution for any $b\in \mathbb{R}^n$. Thus, we have provided another way to prove that if $\|A^{-1}\|<1$ the AVE~\eqref{eq:ave} has a unique solution for all $b\in \mathbb{R}^n$. In addition, from \eqref{eq:fix} we can develop the fixed-point iteration algorithm $x^{k+1}=\mathcal{\tilde{F}}(x^k),\, k=0,1,2,\cdots$ with $\nu \in (0,1)$. This algorithm is convergent if $\|A^{-1}\|<1$. Indeed, when $\nu = \frac{\gamma}{2}$, this fixed-point iteration scheme reduces to the later DRs scheme \eqref{eq:iterI}. 

\section{The  methods and convergence}\label{sec:met}
In this section, we will describe our methods formally and prove their convergence.
The following assumption will be used during the convergence analysis.
\begin{assu}\label{assump} Assume that
\begin{equation}\label{assump}
\framebox{
\parbox{10.8cm}{
$\|A^{-1}\|\le 1$ and the solution set $S^*$ of the AVE~\eqref{eq:ave} is nonempty.
}
}
\end{equation}\end{assu}

This assumption includes $\|A^{-1}\|<1$ or equivalently $\sigma_{\min}(A)>1$ as a special case, from which the AVE~\eqref{eq:ave} has a unique solution for any
$b \in  \mathbb{R}^n$~\cite{mame2006}.

As stated in the previous section, to solve the AVE~\eqref{eq:ave} is equivalent to finding a zero of the residual function $e(x)$, defined as in~\eqref{eq:res}. Since $e(x)=0$ is equivalent to
$$
Q(x) +F(x) = Q(x)+F(x)-\gamma \rho(x)G^{-1}e(x)
$$
with {$\gamma\ne 0$}, $G\in \mathbb{R}^{n\times n}$ being invertible, and $\rho(x)\ne 0$ whenever $e(x)\ne 0$, we can construct the following fixed-point iteration scheme for solving the AVE~\eqref{eq:ave}:
\begin{equation*}
(Q+F)(x^{k+1}) = (Q+F)(x^k)-\gamma \rho(x^k)G^{-1}e(x^k).
\end{equation*}
When $\gamma \in (0,2)$, $G$ is symmetric positive definite and
\begin{equation} \label{eq:rho}
\rho(x)=\frac{\|e(x)\|^2}{e(x)^T G^{-1} e(x)},
\end{equation}
this fixed-point iteration becomes the implicit method  proposed in \cite{he1999},  which for solving the monotone GVI problem~\eqref{eq:gvi}. Essentially, it was developed from the perspective of operator splitting, and when $\gamma = 1$ and $G = I$, it reduces to the Douglas-Rachford algorithm~\cite{dora1956}. {For solving the equivalent GLCP~\eqref{eq:glcp} of the AVE~\eqref{eq:ave}},
our Douglas-Rachford splitting algorithm is as follows, where we have used \eqref{QF} to simplify the representation.

\begin{alg}\label{Alg1}{\bf $(${Exact} Douglas-Rachford splitting method for the
AVE~\eqref{eq:ave}$)$}
Given $x^0 \in \mathbb{R}^n$, $\gamma \in (0,2)$ and a symmetric positive definite matrix $G\in \mathbb{R}^{n\times n}$.
For $k=0,1,\cdots,$ if $x^k \notin S^*$, then
\begin{equation} \label{eq:subproblem}
x^{k+1} \quad  \text{solves} \quad \Theta_k(x)=0,
\end{equation}
where
\begin{equation} \label{eq:Theta}
\Theta_k(x)=2Ax-2Ax^k+\gamma\rho(x^k)G^{-1}(Ax^k-|x^k|-b)
\end{equation}
and $\rho(x)$ is defined as in \eqref{eq:rho}.
\end{alg}

In Algorithm~\ref{Alg1}, each iteration step requires the exact solution of the subproblem \eqref{eq:subproblem}. Indeed, from \eqref{eq:Theta} we have
\begin{equation} \label{eq:iter}
x^{k+1}= x^{k}-\frac{1}{2}\gamma \rho(x^k)A^{-1}G^{-1}(Ax^k-|x^k|-b).
\end{equation}
That is, each iteration step requires the exact solution of the linear system with coefficient matrix $GA$. In particular, let $G=I$, then $\rho(x)\equiv 1$ and the exact  iteration scheme~\eqref{eq:iter} reduces to
\begin{equation} \label{eq:iterI}
x^{k+1}= \left(1-\frac{1}{2}\gamma \right)x^{k}+\frac{1}{2}\gamma A^{-1}(|x^k|+b)
\end{equation}
and each iteration step requires the exact solution of the linear system of equations with coefficient matrix $A$. The iterative scheme \eqref{eq:iterI} is similar to the first stage of the SOR-like iteration method \eqref{eq:equiv}; however, \eqref{eq:iterI} is not dependent on the iteration sequence $\{y^k\}$. Thus, the iteration scheme~\eqref{eq:iterI} is more competitive than the SOR-like iteration method in terms of CPU time, especially for large-scale problems.

The iteration sequence $\{x^k\}$ generated by Algorithm~\ref{Alg1} has the following property.

\begin{thm}\label{thm:iterpro}Assume that \eqref{assump} holds and let $x^*\in S^*$ be an arbitrary solution of the AVE~\eqref{eq:ave}. Then the sequence $\{x^k\}$ generated by Algorithm~\ref{Alg1} satisfies
\begin{equation}\label{exact}
\left\|A(x^{k+1}-x^*)\right\|^2_G \le \left\| A(x^k-x^*) \right \|^2_G-\frac{\gamma(2-\gamma)}{4}\rho(x^k)\left\|e(x^k)\right\|^2.
\end{equation}
\end{thm}
\begin{proof}
Note that $\Theta_k(x^{k+1})=0$ means that
\begin{equation*}
2Ax^{k+1}=2Ax^k-\gamma\rho(x^k)G^{-1}e(x^k).
\end{equation*}
Hence,
\begin{align*}
&\left\|2A(x^{k+1}-x^*)\right\|^2_G\\
&\qquad = \left\| 2A(x^{k}-x^*)-\gamma\rho(x^k)G^{-1}e(x^k) \right \|^2_G\\
&\qquad = \left\| 2A(x^{k}-x^*)\right \|^2_G -4\gamma \rho(x^k)(x^k-x^*)^TA^T e(x^k)+\left(\gamma\rho(x^k)\right)^2 e(x^k)^T G^{-1}e(x^k) \\
&\qquad \le \left\| 2A(x^k-x^*)\right\|^2_G-2\gamma\rho(x^k)\left\|e(x^k)\right\|^2 + \left(\gamma\rho(x^k)\right)^2 e(x^k)^T G^{-1} e(x^k)\\
&\qquad =\left\| 2A(x^k-x^*)\right\|^2_G-\gamma(2-\gamma)\rho(x^k)\left\|e(x^k)\right\|^2,
\end{align*}
where the inequality follows from \eqref{ie:contr} and the last equality from \eqref{eq:rho}.
\end{proof}

Since $G$ is symmetric positive definite,
{ $$
e(x)^TG^{-1}e(x)\le \lambda_{\max}(G^{-1})\|e(x)\|^2.$$
From \eqref{eq:rho}, we have that
$$
\rho(x)\ge \frac{1}{\lambda_{\max}(G^{-1})}=\lambda_{\min}(G)>0,$$}
provided that $x$ is not a solution of AVE~\eqref{eq:ave}.
As a consequence,
\begin{equation*}
\left\|A(x^{k+1}-x^*)\right\|^2_G \le \left\| A(x^k-x^*) \right \|^2_G-\frac{\gamma(2-\gamma)\lambda_{\min}(G)}{4}\left\|e(x^k)\right\|^2,
\end{equation*}
indicating the sequence $\{x^{k}\}$ is Fej\'er monotone with respect to the solution set $S^*$ of the AVE~\eqref{eq:ave}, measured with the norm $\|\cdot\|_{A^TGA}$. Hence, due to the fact that $A$ is nonsingular and $G$ is symmetric positive definite, the global convergence of $\{x^k\}$ to a solution can be easily established, with the help of Lemma \ref{fejer}.

However, in many cases, solving a system of linear equations exactly is either expensive or impossible. On the other hand, there seems to be little justification of the effort required to calculate an accurate solution of $\Theta_k(x)=0$ in each iteration. Hence, a new algorithm adopting approximate solutions of subproblems is much desirable. Before describing the inexact algorithm, we first recall the following error bound theorem for piecewise linear multifunctions, which is established in \cite[Theorem~3.3]{zhng2014}.

\begin{thm}\label{thm:plm}
Let $\mathcal{G}$ be a piecewise linear multifunction. For any $\kappa > 0$, there exists $\mu >0$ such that\\

\hskip 4cm $\hbox{dist}(x,\mathcal{G}^{-1}(0))\le \mu \hbox{dist}(0,\mathcal{G}(x)), \; \forall\, \|x\|<\kappa.\hfill\qed
$
\end{thm}

\medskip

From Theorem \ref{thm:aveequivpe}, solving AVE \eqref{eq:ave} is equivalent to finding zero points of the residual function $e(\cdot)$, { defined as in~\eqref{ex}, } which is piecewise linear. Hence, there is $\mu>0$ such that
\[\label{err}\hbox{dist}(x,S^*)\le \mu \|e(x)\|,\; \forall\, \|x\|<\kappa.\]

In the following, we develop an inexact method, in which we solve $\Theta_k(x)=0$ approximately.

\begin{alg}\label{Alg2}{\bf $($Inexact Douglas-Rachford splitting method for  the AVE \eqref{eq:ave}$)$}
Given $x^0 \in \mathbb{R}^n$, $\gamma \in (0,2)$, $\delta\in(0,1)$, and a symmetric positive definite matrix $G\in \mathbb{R}^{n\times n}$. 
For $k=0,1,\cdots,$ if $x^k \notin S^*$, then find $x^{k+1}$ such that
\begin{equation}\label{ie:crita1}
\left\| \Theta_k(x^{k+1})\right\| \le \alpha_k \left\| e(x^k)\right\|,
\end{equation}
where $\Theta_k(x)$ is defined as in \eqref{eq:Theta} and $\rho(x)$ is defined as in \eqref{eq:rho}, and
{
\[\label{al}
0\le \alpha_k \le \frac{(1-\delta)\gamma(2-\gamma)\rho(x^k)}{4\mu\|A^TG\|+2\gamma\rho(x^k)+\lambda_{\max}(G)}<1.\]
}
\end{alg}


Now we can prove the following contractive property between the generated sequence and the solution set.
\begin{thm}\label{thm:wcov}%
Assume that \eqref{assump} holds and let $x^*\in S^*$ be an arbitrary solution of the
AVE~\eqref{eq:ave}. Then the sequence $\{x^k\}$ generated by Algorithm~\ref{Alg2} satisfies
\begin{equation}\label{inexact}
\left\|A(x^{k+1}-x^*)\right\|^2_G  \le  \left\| A(x^k-x^*) \right \|^2_G -\frac{\delta\gamma(2-\gamma)\rho(x^k)}{4}\left\|e(x^k)\right\|^2,  \; \forall\, k\ge 0.
\end{equation}
\end{thm}
\begin{proof}
Recall that  in Algorithm \ref{Alg2}
\begin{equation*}
2A(x^{k+1}-x^k)=-\gamma \rho(x^k)G^{-1}e(x^k)+\Theta_k(x^{k+1}).
\end{equation*}
Hence,
\begin{align}\nonumber
&\left\|2A(x^{k+1}-x^*)\right\|^2_G\\\nonumber
&\qquad =\left\| 2A(x^k-x^*) -\left[ \gamma \rho(x^k)G^{-1}e(x^k)-\Theta_k(x^{k+1})\right] \right\|^2_G  \\\nonumber
&\qquad =\left\| 2A(x^k-x^*) \right\|^2_G  + \left\| \gamma \rho(x^k)G^{-1}e(x^k)-\Theta_k(x^{k+1})\right\|^2_G  \\\nonumber
&\qquad\qquad -4\gamma \rho(x^k)(x^k-x^*)^TA^Te(x^k) +4(x^k-x^*)^TA^T G \Theta_k(x^{k+1}) \\\nonumber
&\qquad  \le \left\|2A(x^k-x^*) \right\|^2_G -2\gamma \rho(x^k) \left\| e(x^k)\right\|^2 \\ \label{eqr11}
&\qquad\qquad + \left\| \gamma \rho(x^k)G^{-1}e(x^k)-\Theta_k(x^{k+1})\right\|^2_G +4(x^k-x^*)^TA^T G \Theta_k(x^{k+1}),
\end{align}
where the inequality follows from Theorem~\ref{thm:contrac}.
We now deal with the last two terms in \eqref{eqr11} one by one.

For the first term, we have
\begin{align}\nonumber
&\left\| \gamma \rho(x^k)G^{-1}e(x^k)-\Theta_k(x^{k+1})\right\|^2_G \\\nonumber
&\qquad =\left\| \gamma \rho(x^k)G^{-1}e(x^k)\right\|^2_G+ \left\|\Theta_k(x^{k+1})\right\|^2_G - 2\gamma \rho(x^k)\Theta_k(x^{k+1})^T e(x^k) \\\nonumber
&\qquad \le \left\| \gamma \rho(x^k)G^{-1}e(x^k)\right\|^2_G+ \left\|\Theta_k(x^{k+1})\right\|^2_G +2\gamma \rho(x^k)\left\|\Theta_k(x^{k+1})\right\| \left\|e(x^k)\right\|\\\label{eqr13}
&\qquad \le \gamma^2 \rho(x^k) \left\|e(x^k)\right\|^2 +\lambda_{\max}(G) \alpha_k^2  \left\|e(x^k)\right\|^2
+ 2\gamma \rho(x^k) \alpha_k \left\|e(x^k)\right\|^2,
\end{align}
where the first inequality follows from the Cauchy-Schwarz inequality, and the last one from the definition of $\rho$
and the approximate criterion \eqref{ie:crita1}.

We now consider the second term. Using again the Cauchy-Schwarz inequality and  the approximate criterion \eqref{ie:crita1}, we have
\begin{align}\label{eqr12}
&4(x^k-x^*)^TA^TG \Theta_k(x^{k+1}) \le4\mu\alpha_k\left\|A^TG\right\|\left\| e(x^k)\right\|^2.
\end{align}

Substituting \eqref{eqr12} and \eqref{eqr13} in inequality  \eqref{eqr11}, we have
\begin{align*}
&\left\|2A(x^{k+1}-x^*)\right\|^2_G\\
&\qquad \le \left\|2A(x^k-x^*)\right\|^2_G- \gamma(2-\gamma)\rho(x^k) \left\|e(x^k)\right\|^2\\
&\qquad\qquad +\left(\lambda_{\max}(G) \alpha_k^2+ 2\gamma \rho(x^k)\alpha_k+4\mu\alpha_k\left\|A^TG\right\|\right)\left\|e(x^k)\right\|^2\\
&\qquad \le \left\|2A(x^k-x^*)\right\|^2_G- \delta\gamma(2-\gamma)\rho(x^k) \left\|e(x^k)\right\|^2,
\end{align*}
where the last inequality follows from the { determination} of $\alpha_k$.
\end{proof}

\begin{rem}
Comparing the contractive properties of the sequences generated by Algorithm~\ref{Alg1} and Algorithm~\ref{Alg2}, i.e.,  the inequalities \eqref{exact} and \eqref{inexact}, we find that the only difference between them is the parameter $\delta$ in \eqref{inexact}. Essentially, if $\delta=1$, then \eqref{inexact} reduces to \eqref{exact}; and for this case, $\alpha_k\equiv 0$, indicating the algorithm also reduces to the exact one. Hence, in the following, we only analyze the convergence behavior of Algorithm~\ref{Alg2}; and that for Algorithm~\ref{Alg2} can be derived in a similar way. \qed
\end{rem}


We are now ready to present the convergence of Algorithm \ref{Alg2}.

\begin{thm}\label{thm:cov} 
Assume that \eqref{assump} holds. Then the sequence $\{x^k\}$ generated by Algorithm~\ref{Alg2} globally converges to
a solution of the AVE~\eqref{eq:ave}. Furthermore, the rate of convergence is linear.
\end{thm}
\begin{proof} From \eqref{eq:rho} we conclude that { $\rho(x^k)\ge 1/\lambda_{\max}(G^{-1})$ } for all $k$. Hence, it follows from
the nonsingularity of $A$ that $\{x^k\}$ is bounded which does have at least one cluster point, denoted by $\hat x$. Moreover, \eqref{inexact} also implies that
$$
\lim_{k\to\infty}\|e(x^k)\|=0,$$
which together with the continuity of the mapping $e(\cdot)$, implies that $e(\hat x)=0$ and hence $\hat x$ is a solution of
AVE \eqref{eq:ave}.
%
Since $x^*$ is an arbitrary solution in \eqref{inexact}, we can set $x^*:=\hat x$ and obtain
\begin{eqnarray}\label{linear}
\left\| A(x^{k+1}-\hat x) \right \|^2_G\le  \left\|A(x^k-\hat x) \right \|^2_G-\frac{\delta\gamma(2-\gamma)}{4\lambda_{\max}(G^{-1})}\|e(x^k)\|^2.
\end{eqnarray}
It follows from \eqref{err} that
$$
\| A(x^{k+1}-\hat x) \|^2_G\le\lambda_{\max}(A^TGA)\|x^{k}-\hat x\|^2\le \mu \lambda_{\max}(A^TGA)\|e(x^{k})\|^2.$$
Substituting the last inequality into \eqref{linear} and then rearranging terms, we get
\[\label{linear2}
\left\| A(x^{k+1}-\hat x) \right \|^2_G\le  \varrho\left\|A(x^k-\hat x) \right \|^2_G,\]
where
{
$$
\varrho:=\frac{4\mu \lambda_{\max}(G^{-1}) \lambda_{\max}(A^TGA)}{\delta\gamma(2-\gamma) +4\mu \lambda_{\max}(G^{-1}) \lambda_{\max}(A^TGA)}\in (0,1).$$
}
\eqref{linear2} and the nonsingularity of $A$ then imply $\{x^k\}$ converges globally and in a linear manner to a solution of the AVE \eqref{eq:ave}.
\end{proof}

%
Since we are considering the case where $\|A^{-1}\|\le 1$, which can not guarantee the uniquely solvable of the
AVE \eqref{eq:ave} if $\|A^{-1}\|= 1$. After proving the globally linear convergence of the proposed algorithms for the case that the solution set is nonempty, we now turn attention to the case where the solution set is empty, i.e., $S^*=\emptyset$.
\begin{thm}\label{div}Assume that { $\|A^{-1}\|= 1$} and the AVE \eqref{eq:ave} does not possess a solution. Then the sequence $\{x^k\}$ generated by Algorithm~\ref{Alg2} diverges to infinity, i.e., $\|x^k\|\to \infty$ as $k\to\infty$.
\end{thm}
\begin{proof}
We prove the assertion by contradiction. Suppose that on the contrary, the sequence $\{x^k\}$ generated by Algorithm~\ref{Alg2} is bounded. Then there is some $M>0$ such that $\|x^k\|\le M$. Define
$$
\Xi:=\{x\in\R^n\;|\; \|x\|\le M+1\},$$
and consider the constrained { AVE } problem of finding $x\in\Xi$ such that
\begin{equation}\label{cave}
    Ax-|x|-b=0.
\end{equation}
Using the identity $|x|=2\max\{0,x\}-x$, the problem can be reformulated to the problem of finding
$x\in\Xi$ such that
{ $$
x=2\max\{0,x\}-Ax+b,$$ }
i.e., finding a fixed point of the continuous mapping {  $h(x):=2\max\{0,x\}-Ax+b$ } on the compact set $\Xi$.
Using the  celebrated Brouwer fixed-point theorem for a continuous map, we conclude that the constrained AVE problem
\eqref{cave} possesses a solution. Now consider Algorithm \ref{Alg2} applying to the constrained AVE problem
\eqref{cave}, the procedure should be the same except the constrained that $x\in\Xi$. More precisely, at iteration $k$, the next iteration is generated by finding $x^{k+1}\in\Xi$, such that
\begin{equation*}
\left\| \Theta_k(x^{k+1})\right\| \le \alpha_k \left\| e(x^k)\right\|,
\end{equation*}
where $\Theta_k(x)$ is defined as in \eqref{eq:Theta} and $\rho(x)$ is defined as in \eqref{eq:rho}, and
$\alpha_k$ is { determined } by \eqref{al}.  From Theorem \ref{thm:cov}, $\{x^k\}$ converges to a solution of the constrained
AVE \eqref{cave}, which should be a solution of the AVE \eqref{eq:ave}, a contradiction to the assumption that the AVE \eqref{eq:ave} does not have a solution. This completes the proof.
\end{proof}

We now borrow a trivial example from \cite{mang2009} to illustrate the behavior of the proposed algorithm when
the AVE \eqref{eq:ave} does not have a solution. Consider the AVE in $R^1$: $x-|x|-1=0$. This AVE has no solution and satisfies $\|A^{-1}\|=1$. For this example, according to~\eqref{eq:iterI}, $x^{k+1}=x^k+\frac{1}{2}\gamma (|x^k|-x^k)+\frac{1}{2}\gamma=\cdots = x^0+\frac{1}{2}\gamma (|x^0|-x^0)+\frac{k+1}{2}\gamma \rightarrow +\infty$ as $k\rightarrow +\infty$, for any $\gamma\in (0,2)$.


\section{Numerical results}\label{sec:numer}
In this section, we apply the proposed algorithms to solving some concrete examples. We compare the new algorithms with some existing algorithms. The purpose here is to show that, on one hand, the new algorithms have the largest range of convergence, i.e., besides the condition that $\|A^{-1}\|\le 1$ and the solution set is nonempty, there is no other requirements; On the other hand and more importantly, { the new algorithms are much more efficient than those existing.}

As we have reviewed in the first section, recently, there has been a surge of interest in solving the AVE~\eqref{eq:ave} with $\|A^{-1}\|<1$ and a large number of numerical iteration methods have been proposed, including the generalized Newton method \cite{mang2009,lilw2018}, the smoothing Newton method\cite{caqz2011}, the inexact semi-smooth Newton method \cite{crfp2016}, the Levenberg-Marquardt method \cite{iqia2015}, the SOR-like iteration method \cite{kema2017,guwl2019,chyh2020}, the fixed point iteration method \cite{ke2020} and others; see \cite{edhs2017,abhm2018,mang2007,maer2018,maee2017,miys2017,noin2012,
sayc2018,wacl2017,mang2007a,guhl2017} and references therein. We describe three of them in a little detail, since they will be tested in our numerical experiments.

\begin{alg}[\!\!\cite{mang2009}]\label{alg:newton}{\bf $($Exact semi-smooth Newton method$)$} Given an
            initial guess $x^0\in \mathbb{R}^n$, the iteration sequence $\{x^k\}$ is generated by
           \begin{equation}\label{eq:newton}
               \left[A-D(x^k)\right]x^{k+1} = b
            \end{equation}
             for $k = 0,\,1,\,\cdots$ until convergence, where
\end{alg}

\begin{alg}[\!\!\cite{crfp2016}]\label{alg:ienewton}{\bf $($Inexact semi-smooth Newton method$)$} Given an
            initial guess $x^0\in \mathbb{R}^n$, find some step $x^{k+1}$ which satisfies
            \begin{equation*}
               \left[A-D(x^k)\right]x^{k+1} = b+r_k,\quad \text{where} \quad \left\|  r_k\right\|
               \le \theta \left\|Ax^k-|x^k|-b\right\|
            \end{equation*}
             for $k = 0,\,1,\,\cdots$ until convergence. Here, $\theta \ge 0$ is the residual
             relative error tolerance.
\end{alg}

\begin{alg}[\!\!\cite{kema2017}]\label{alg:sorl}{\bf $($SOR-like iteration method$)$} Given initial
           guesses $x^0,\, y^0 \in \mathbb{R}^n$,  for $k = 0,\,1,\,\cdots$ until
           convergence, compute
           \begin{align}\label{eq:sorl}
                \left\{\begin{array}{l}
                   x^{k+1}=(1-\omega)x^{k}+\omega A^{-1}(y^{k}+b),\\
                   y^{k+1}=(1-\omega)y^{k}+\omega |x^{k+1}|,
                   \end{array}\right.
           \end{align}
           where $\omega>0$ is the iteration parameter.
\end{alg}

Note that in all these algorithms, the main task per iteration is solving a system of linear equations. The main difference between the Newton-type methods and the SOR-like algorithm, as well as  the algorithms proposed in this paper, is that the coefficient matrices: In Newton-type methods, the matrix is $A-D(x^k)$, which is varying along with the iteration; while in the SOR-like method and Algorithms \ref{Alg1}-\ref{Alg2},  the matrix is $A$ and $GA$, respectively, a fixed one.

Five algorithms will be tested and we now list the parameters utilized in them during our experiments.

\begin{enumerate}

 \item DRs: the exact Douglas-Rachford splitting method, i.e.,
          Algorithm~\ref{Alg1} with $G = I$ and $\gamma = 1.98$.

  \item InexactDRs: the inexact Douglas-Rachford splitting method, i.e.,
          Algorithm \ref{Alg2} with $G=I$ and $\gamma = 1.98$. The inexact criterion
          \eqref{ie:crita1} is used with {  $\alpha_k = \min\left\{1,\frac{1}{\max\{1,k-k_{\max}\}}\right\}$}
          and $k_{\max}=10$.

  \item Newton: the exact semi-smooth Newton method \cite{mang2009}, i.e.,
           Algorithm~\ref{alg:newton}.

  \item InexactNewton: the inexact semi-smooth Newton method proposed in
           \cite{crfp2016}, i.e., Algorithm~\ref{alg:ienewton}. Since we are interested in
           testing methods which have theoretical guarantee, in the following, we test
           Newton and InexactNewton methods only in the case that
           $\|A^{-1}\|<\frac{1}{3}$ is imposed (see Example~\ref{ex:exam1}). In addition, as defined in \cite{crfp2016},
           \begin{equation}\label{eq:theracr}
              \theta =0.9999\cdot \frac{1-3\|A^{-1}\|}{\|A^{-1}\|(\|A\|+3)}
          \end{equation}
          is used. If $\frac{1}{3}<\|A^{-1}\|<1$, this $\theta$ is negative and it can not be used
          anymore and this is another reason why we will not test the InexactNewton
          method in the case $\|A^{-1}\|<1$ (see Example~\ref{ex:exam2}). In addition, in the following Example~\ref{ex:exam1}, since the singular values of $A$ are known, the $\theta$ defined as in \eqref{eq:theracr} can be used with little cost. However, in practice, though the upper bound for $\theta$ in \eqref{eq:theta} $($for the norms $\|\cdot\|_1$ or $\|\cdot\|_{\infty}$$)$ is easy to compute, it will still be expensive  for large $n$ to use $\theta$ defined as in \eqref{eq:theracr}. Inversely, the inexact criterion for the InexactDRs method is easy to implement.

  \item SOR-like: the exact SOR-like iteration method, which was proposed in
           \cite{kema2017} and further studied in \cite{guwl2019,chyh2020}, i.e.,
           Algorithm~\ref{alg:sorl}. If $\|A^{-1}\|\le \frac{1}{4}$, the optimal iteration parameter $\omega^*=1$. Since it will be expensive to compute the optimal iteration parameter when $\frac{1}{4}<\|A^{-1}\|<1$ \cite{chyh2020}, we simply let $\omega = 0.9$ in this paper, except Example~\ref{ex:sorvsdrs}.

\end{enumerate}

From \eqref{eq:newton}, each step of the Newton method requires the exact solution of a linear system with coefficient matrix $A-D(x^k)$. Since $D(x^k)$ may vary with $x^k$, it  simply uses the usual Gaussian elimination with partial pivoting, such as MATLAB's operator ``$\backslash$", to carry out the inversion, with a cost at $O(n^3)$.
According to \eqref{eq:iterI} and \eqref{eq:sorl}, each step of the DRs method and the SOR-like method requires the exact solution of a linear system with constant coefficient matrix $A$. Solving systems $Ax=d^k$ for different $k$ only requires that $A$ be factored once, at a cost of $O(n^3)$ flops for the worst case ($A$ is dense); while per iteration, its cost reduces to $O(n^2)$ flops (matrix-vector products). For such a system of linear equations with fixed coefficient matrix $A$, we can adopt the intrinsic properties of $A$, especially when $A$ is large and sparse, and here we adopt the LU decomposition.  Hence, If there are $K$ iteration steps, the cost is $O(n^3)+KO(n^2)$.  Specifically, we use
 $
 dA ={ \bf decomposition}(A,\text{`lu'})
 $
 to generate the LU decomposition of $A$, where ${\bf decomposition}$ is the routine in MATLAB. $dA ={ \bf decomposition}(A,\text{`lu'})$ returns the LU decomposition of matrix $A$, which can be used to solve the linear system $Ax=b$ efficiently. The call $x=dA\backslash b$ returns the same vector as $A\backslash b$, but is typically faster. Finally, as suggested for the InexactNewton method in \cite{crfp2016}, we  use the LSQR algorithm \cite{pasa1982} as the inner iteration method for the InexactDRs method.

All computations are done in MATLAB R2017b with a machine precision $2.22\times 10^{-16}$ on a personal computer with 2.60GHz central processing unit (Intel Core i7), 16GB memory and MacOS operating system. 
For all methods, we  stop its running  if $\|e(x^k)\|\le \epsilon$ and $\epsilon =10^{-8}$ during the implementation. We set $x^0 = -100 + 200 \times \textbf{rand}(n,1)$ and $y^0=x^0$ for the SOR-like method, where $\textbf{rand}$ is the MATLAB function which returns a pseudorandom scalar drawn from the standard uniform distribution on the open interval $(0,1)$. We first generate $x^* = -100 + 200 \times \textbf{rand}(n,1)$  and set $b = Ax^*-|x^*|$, { except Example~\ref{ex:sorvsdrs}.}

For comparison, we  plot the performance profile graphics \cite{domo2002}, and the performance measurement is the running CPU time; namely, we use the performance ratio
$$
r_{p,s} = \frac{t_{p,s}}{\min\{t_{p,s}: s\in \mathcal{S}\}},
$$
where $t_{p,s}$ defined as the mean of the CPU time by using method $s$ to solve problem $p$ five times and $\mathcal{S}$ is the set of corresponding solvers. We set $r_{p,s}=r_M=20$ if and only if solver $s$ does not successfully solve problem $p$ within $50$ iterations. In this manner, the performance profile graphics allow us to compare efficiency and robustness of the tested methods \cite{crfp2016}. Indeed, in a performance profile, efficiency and robustness can be accessed on the extreme left (when $\tau = 1$ in our plots) and right of the graphic, respectively. In addition, in this paper, we will focused on dealing only with the relatively well-conditioned matrices $A$ since all methods presented better robustness on this class of problems; see Figure~\ref{fig:condA} for more detail. The same phenomenon is reported in \cite{crfp2016}.
\begin{figure}[t]
{\centering
\begin{tabular}{ccc}
\hspace{-0.3 cm}
\resizebox*{0.50\textwidth}{0.26\textheight}{\includegraphics{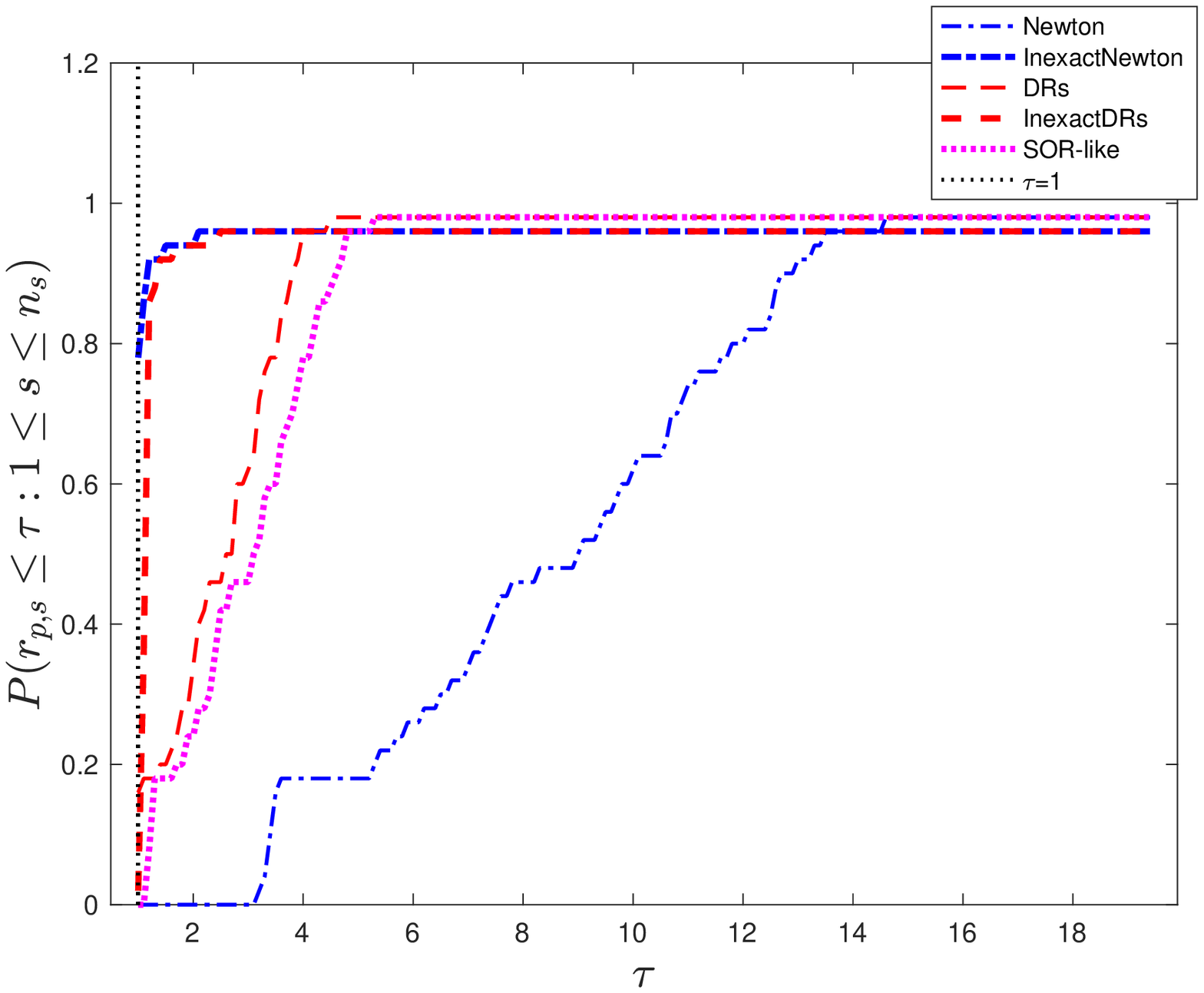}}
& & \hspace{-0.9 cm}
\resizebox*{0.50\textwidth}{0.26\textheight}{\includegraphics{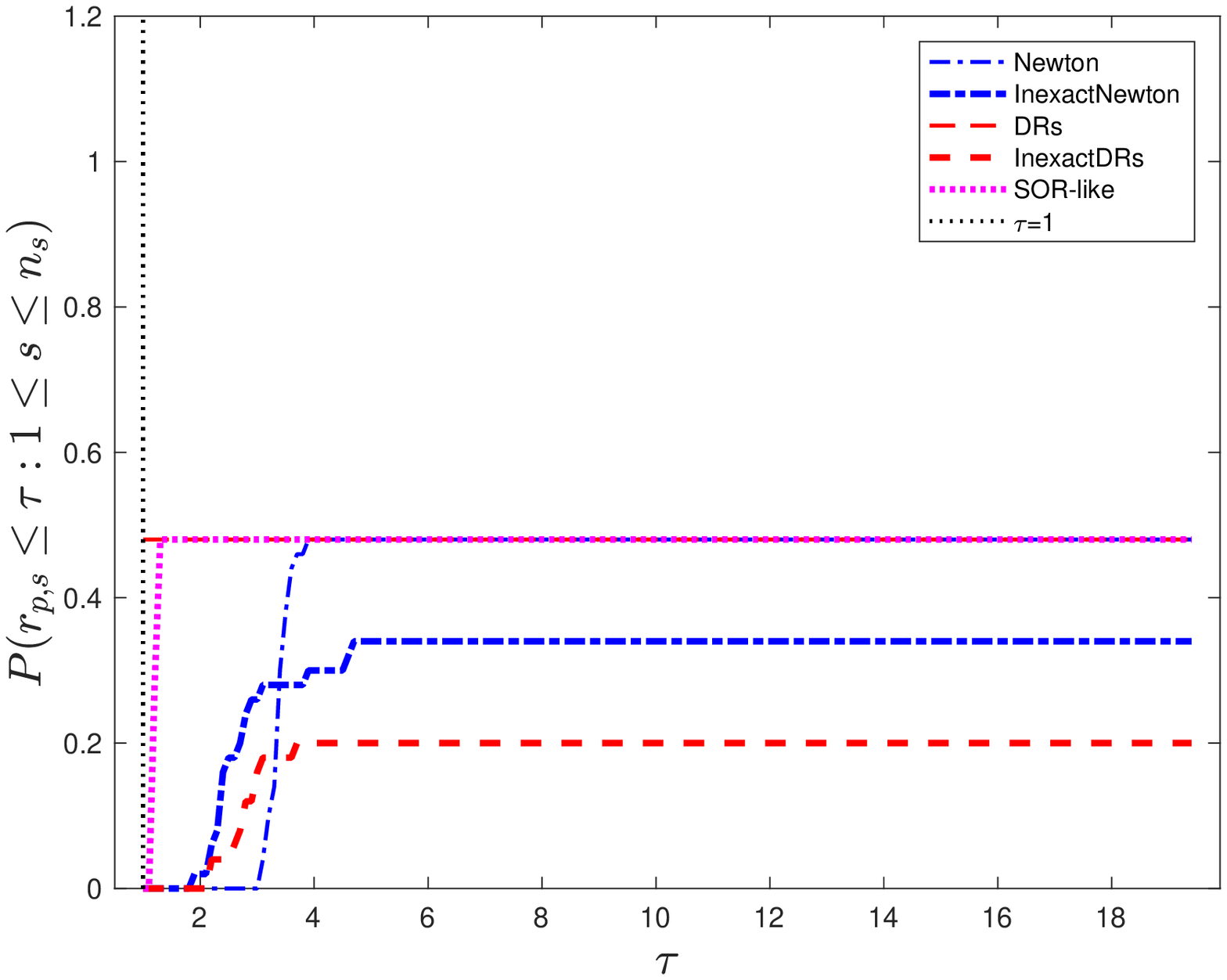}}
\end{tabular}\par
}\vspace{-0.15 cm}
\caption{Performance profile graphics for $50$ AVEs~\eqref{eq:ave} with $n = 10000$ and $\text{density} \approx 0.3$ (left: the average condition number of $A$ is approximately $230.6555$; right: the average condition number of $A$ is approximately $1951.3$).}
\label{fig:condA}
\end{figure}

\begin{exam}[\!\!\cite{guwl2019}]\label{ex:sorvsdrs} Consider the AVE \eqref{eq:ave} with
$$ A = \textbf{tridiag}(-1,8,-1)=\left[\begin{array}{cccccc}
                                     8&-1&0&\cdots&0&0\\
                                     -1&8&-1&\cdots&0&0\\
                                     0&-1&8&\cdots&0&0\\
                                     \vdots&\vdots&\vdots&\ddots&\vdots&\vdots\\
                                     0&0&0&\cdots&8&-1\\
                                     0&0&0&\cdots&-1&8\end{array}\right]\in \mathbb{R}^{n\times n}\quad\text{and}\quad b=Ax^*-|x^*|,$$
where $x^*=[-1,1,-1,1,\cdots,-1,1]^T\in \mathbb{R}^n.$
\end{exam}
For this example, we always have $\|A^{-1}\|=0.1667<0.25$ and $\omega^*=1$ is used. This example is used to compare the performance of the DRs method and the SOR-like method. Numerical results are reported in Tabel~\ref{table-sorvsdrs}, where ``IT", ``CPU'' and ``RES'' denote the number of iterations, the CPU time in seconds, and the residual error defined by $\|e(x^k)\|$, respectively. From Table~\ref{table-sorvsdrs}, we find that each step of the DRs method is less expensive than that of the SOR-like method (the trend is clearer as $n$ becomes larger), which is mainly due to the fact that, even if the cost in updating the auxiliary variable $y$ is not too high, it takes some weights in the total cost, especially when the cost in updating the original variable $x$ is low.
\setlength{\tabcolsep}{5.0pt}
\begin{table}[!h]\small
\centering
\caption{Numerical results for Example~\ref{ex:sorvsdrs}.}\label{table-sorvsdrs}
\begin{tabular}{|c|c|ccccc|}\hline
 \multirow{2}{*}{Method}  &  & \multicolumn{5}{|c|}{$n$}  \\ \cline{3-7}
  &     &  $16000$ & $20000$ & $24000$ & $30000$ & $40000$ \vextra\\ \hline\hline
\multirow{3}{*}{SOR-like} & IT  & $15$ & $15$ & $15$ & $15$& $15$ \\
                 & CPU & $0.0293$ & $0.0406$ & $0.0485$ & $0.0612$& $0.2019$ \\
           & RES & $2.15\times 10^{-9}$ & $2.41\times 10^{-9}$  & $2.63\times 10^{-9}$ &  $2.92\times 10^{-9}$& $3.39\times 10^{-9}$\\
 \multirow{3}{*}{DRs} & IT  & $15$ & $15$ & $15$ & $15$& $15$\\
                 & CPU & $\textbf{0.0251}$ & $\textbf{0.0347}$ & $\textbf{0.0406}$ & $\textbf{0.0489}$& $\textbf{0.0647}$\\
            & RES &$5.53\times 10^{-9}$ & $6.21\times 10^{-9}$ & $6.76\times 10^{-9}$ & $7.52\times 10^{-9}$& $8.71\times 10^{-9}$\\\hline
\end{tabular}
\end{table}

\begin{exam}\label{ex:exam1}
Consider the AVEs \eqref{eq:ave} with $n=10000$, $n = 3000$ and $n = 1500$, respectively, and for each $n$ we generate $200$ AVEs according to the way proposed in \cite{crfp2016}. For all these AVEs, we have $\|A^{-1}\|<\frac{1}{3}$ and thus the Newton and InexactNewton methods are tested in this example. When $n = 10000$, the mean density of $A$ approximately equal to $0.004$ and the average condition number of $A$ is approximately $352.6557$ $($the smallest and largest value is $51.1698$ and $1871.4$, respectively$)$. When $n = 3000$, two sets of $200$ AVEs are generated, the density of the matrices $A$ in the first set is approximately equal to $0.1$ while it is $0.4$ for the second one. In addition, the average condition number of $A$ is approximately $29.5422$ $($the smallest and largest value is $10.6772$ and $113.7403$, respectively$)$ for the first set and $44.5766$ $($the smallest and largest value is $1.7661$ and $281.1603$, respectively$)$ for the second set. When $n = 1500$, two sets of $200$ AVEs are generated, the density of the matrices $A$ in the first set is approximately equal to $0.8$ while it is approximately equal to $1$ for the second one. In addition, the average condition number of $A$ is approximately $16.9338$ $($the smallest and largest value is $5.3346$ and $69.4724$, respectively$)$ for the first set and $16.6732$ $($the smallest and largest value is $5.3193$ and $69.5532$, respectively$)$ for the second set.
\end{exam}

Numerical results are reported in Table~\ref{table-exam1} and Figure~\ref{fig:fig-for-exam1}. From Table~\ref{table-exam1}, we can find that the InexactNewton method has the most wins when $n = 10000$ and $n = 3000$ and that the probability that the InexactNewton method is the winner  is about $0.665$, $0.885$ and $0.485$, respectively, for $n=10000$, $n=3000$ with $\text{density} \approx 0.1$ and $n=3000$ with $\text{density} \approx 0.4$. In addition, when $n=1500$, the Newton method has the most wins for solving the tested problems and that the probability that the Newton method is the winner  is about $0.815$ and $0.74$, respectively, for $\text{density} \approx 0.8$ and $\text{density} \approx 1$. Moreover, the DRs method is superior to the SOR-like method and the Newton method in terms of CPU time when $n=10000$ and $n=3000$. When $n=10000$, the DRs method and the SOR-like method have the highest robustness ($88\%$) while the InexactDRs method has the lowest robustness ($82\%$). When $n = 3000$, the robustness of all methods is considerably high for this example (the range is from $97\%$ to $100\%$) and the two inexact methods perform very well. When $n=1500$, the Newton method has the lowest robustness while the two inexact methods have the highest robustness. Figure~\ref{fig:fig-for-exam1} clearly shows the performance for all methods. In particular, when $n = 10000$, if we choose being within a factor of $0.2$ of the best solver as the scope of our interest, then the InexactDRs method would suffice and the DRs method and the SOR-like method become much more competitive if we extend our $\tau$ of interest to $7$ (more clearly, see the top-right of Figure~\ref{fig:fig-for-exam1}).

\setlength{\tabcolsep}{5.0pt}
\begin{table}[!h]\small
\centering
\caption{Numerical results for Example~\ref{ex:exam1}.}\label{table-exam1}
\begin{tabular}{|l|l|lllll|}\hline
 \multirow{3}{*}{Method}  &  & \multicolumn{5}{|c|}{$n$}  \\ \cline{3-7}
     &   & \multicolumn{1}{l}{$10000$} & \multicolumn{2}{l}{$3000$} & \multicolumn{2}{l|}{$1500$} \\\cline{2-3}\cline{4-5}\cline{6-7}
     &   Density($\%$) &  $0.38$ & $9.51$ & $38.01$ & $76.02$ &  $95.01$\vextra\\ \hline\hline
\multirow{2}{*}{Newton} & Efficiency($\%$)  & $0$ & $0$ & $0$ & $\textbf{81.5}$  & $\textbf{74}$\\
                 & Robustness($\%$) & $86.5$ & $97$ & $97$ & $81.5$ &$74$\\\hline
\multirow{2}{*}{InexactNewton} & Efficiency($\%$)  & $\textbf{66.5}$ & $\textbf{85.5}$ & $\textbf{48.5}$ & $4.5$ &$0.5$\\
                 & Robustness($\%$) & $87.5$ & $\textbf{99}$ & $\textbf{100}$ & $\textbf{99.5}$ &$\textbf{100}$\\\hline
\multirow{2}{*}{DRs} & Efficiency($\%$)  & $19$ & $12.5$ & $47.5$ & $7$ &$23.5$ \\
                 & Robustness($\%$) & $\textbf{88}$ & $97$ & $97.5$ & $98$ &$\textbf{100}$\\\hline
\multirow{2}{*}{InexactDRs} & Efficiency($\%$)  & $2.5$ & $1$ & $3$ & $6$ &$2$\\
                 & Robustness($\%$) & $82$ & $\textbf{99}$ & $\textbf{100}$ & $\textbf{99.5}$ &$\textbf{100}$\\\hline
\multirow{2}{*}{SOR-like} & Efficiency($\%$)  & $0.5$ & $0$ & $0$ & $0$ &$0$\\
                 & Robustness($\%$) & $\textbf{88}$ & $97$ & $97.5$ & $98$ &$\textbf{100}$\\\hline
\end{tabular}
\end{table}

\begin{figure}[t]
{\centering
\begin{tabular}{ccc}
\hspace{-0.3 cm}
\resizebox*{0.50\textwidth}{0.26\textheight}{\includegraphics{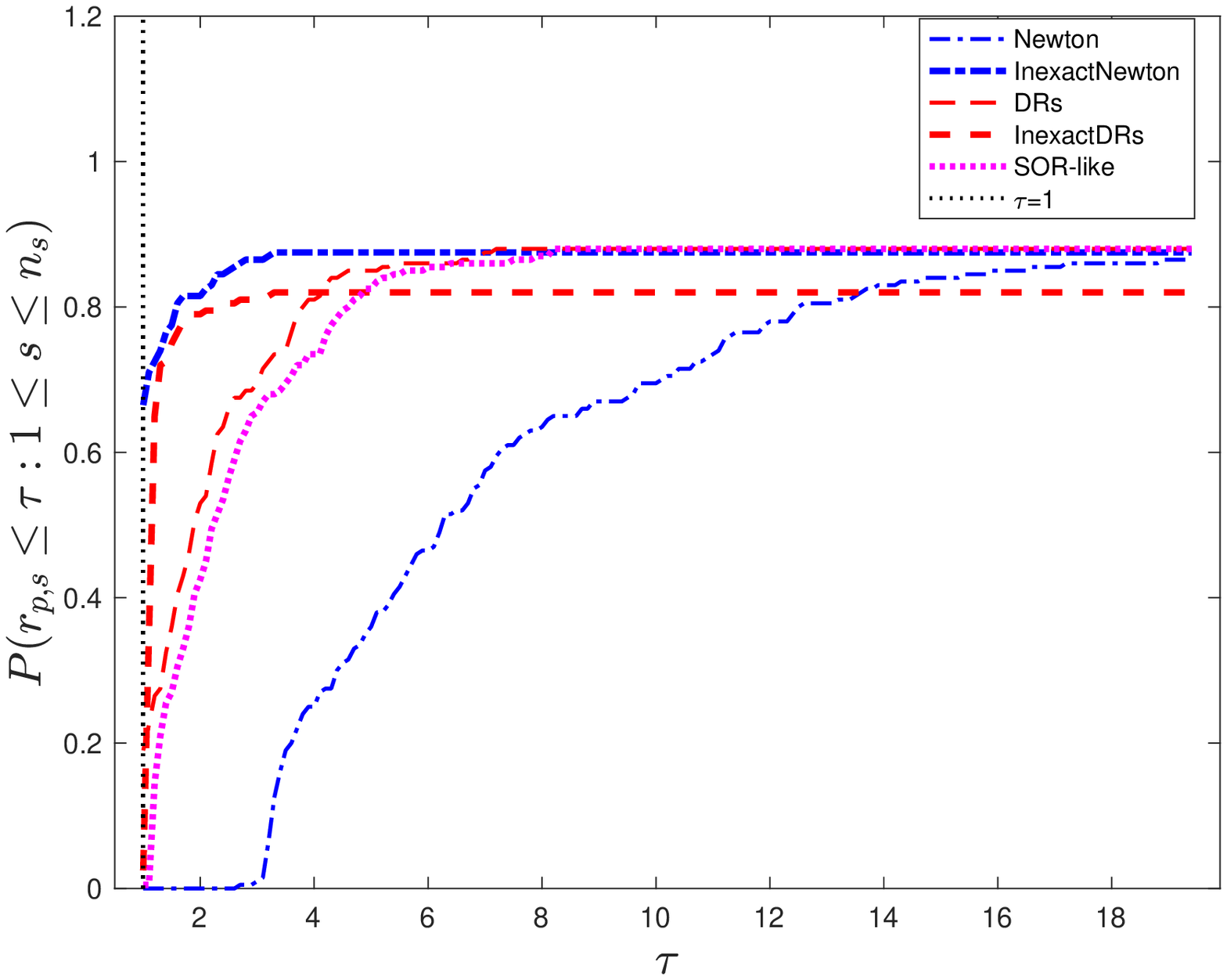}}
& & \hspace{-0.9 cm}
\resizebox*{0.50\textwidth}{0.26\textheight}{\includegraphics{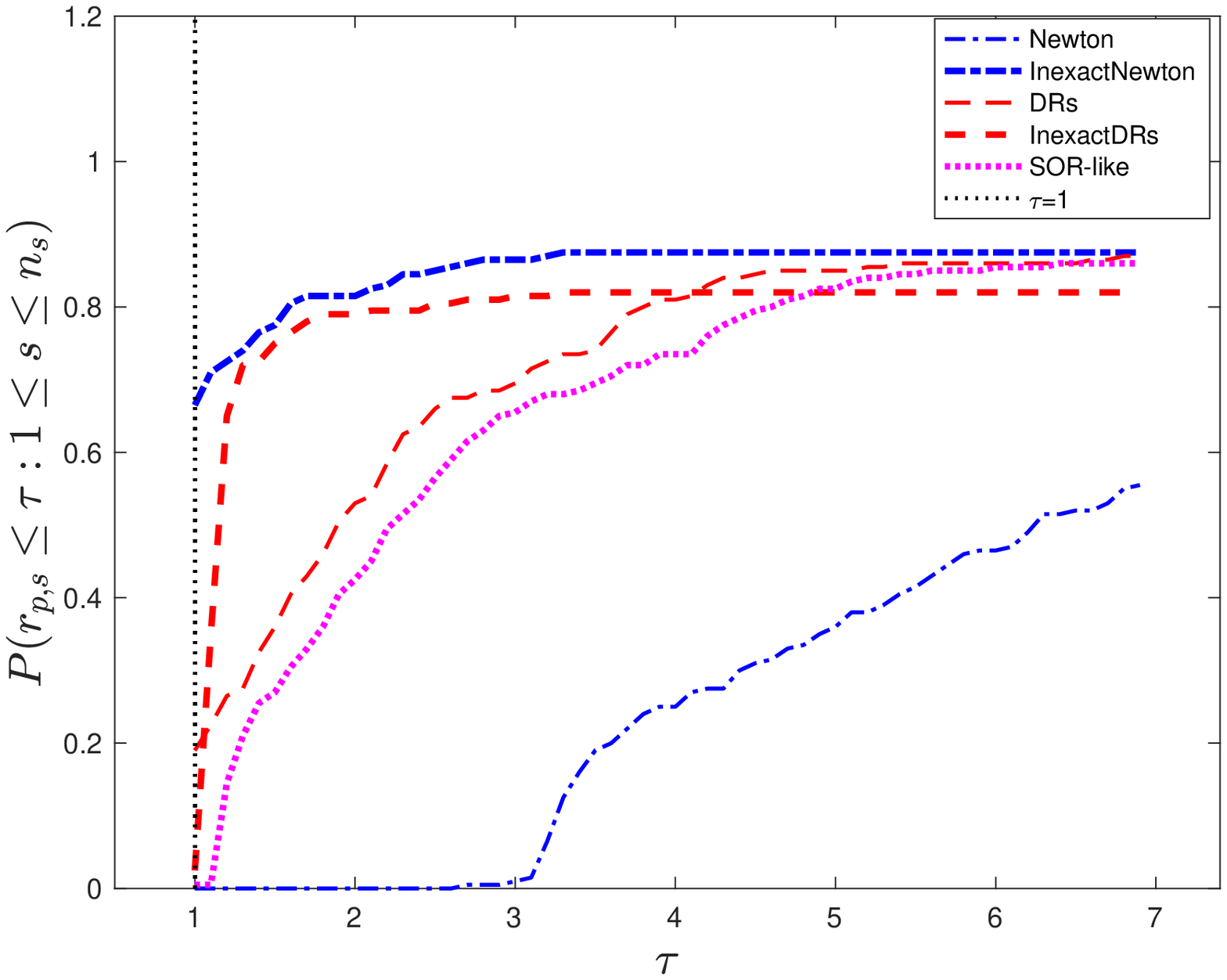}} \vspace{2ex}\\
\hspace{-0.3 cm}
\resizebox*{0.50\textwidth}{0.26\textheight}{\includegraphics{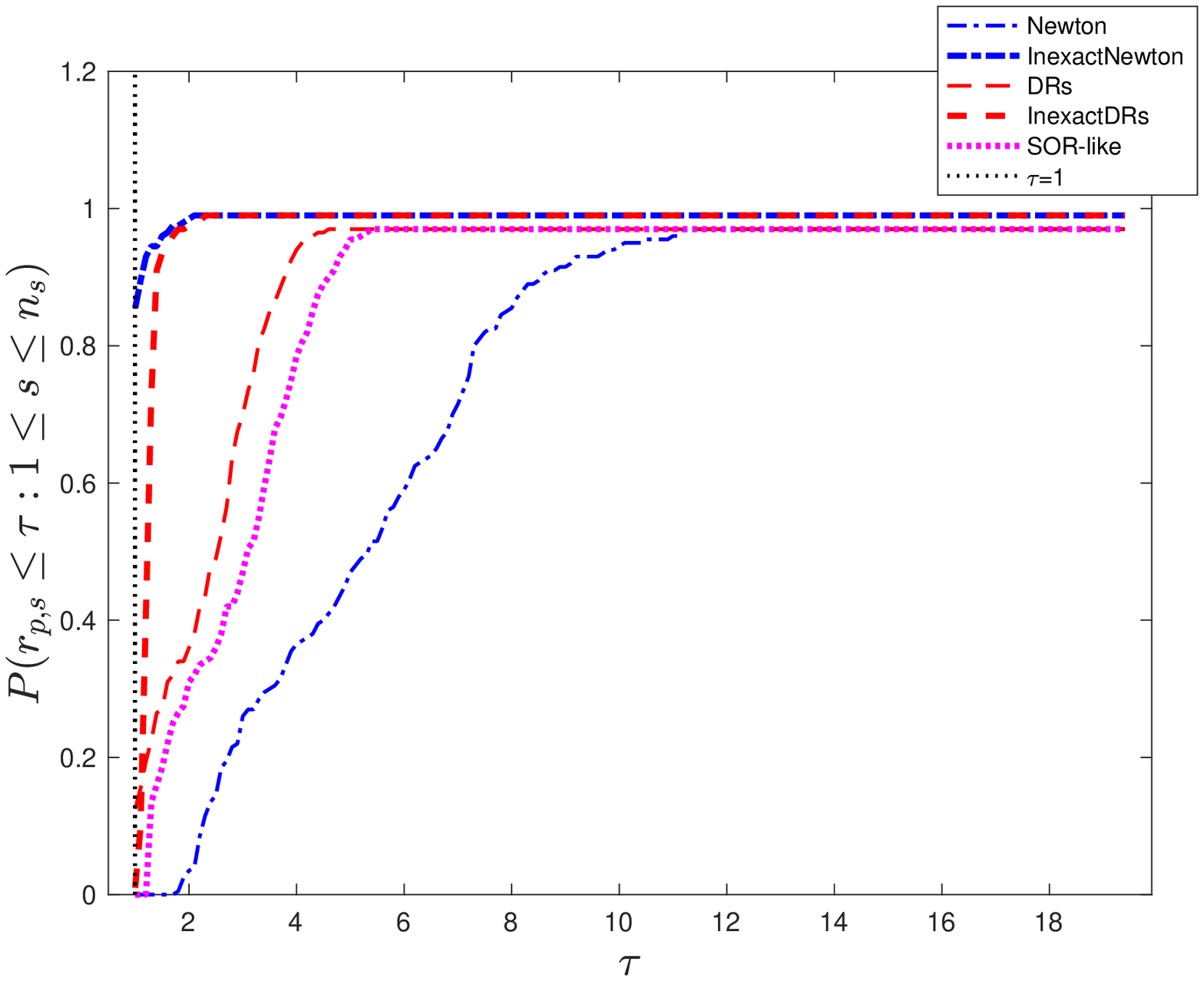}}
& & \hspace{-0.9 cm}
\resizebox*{0.50\textwidth}{0.26\textheight}{\includegraphics{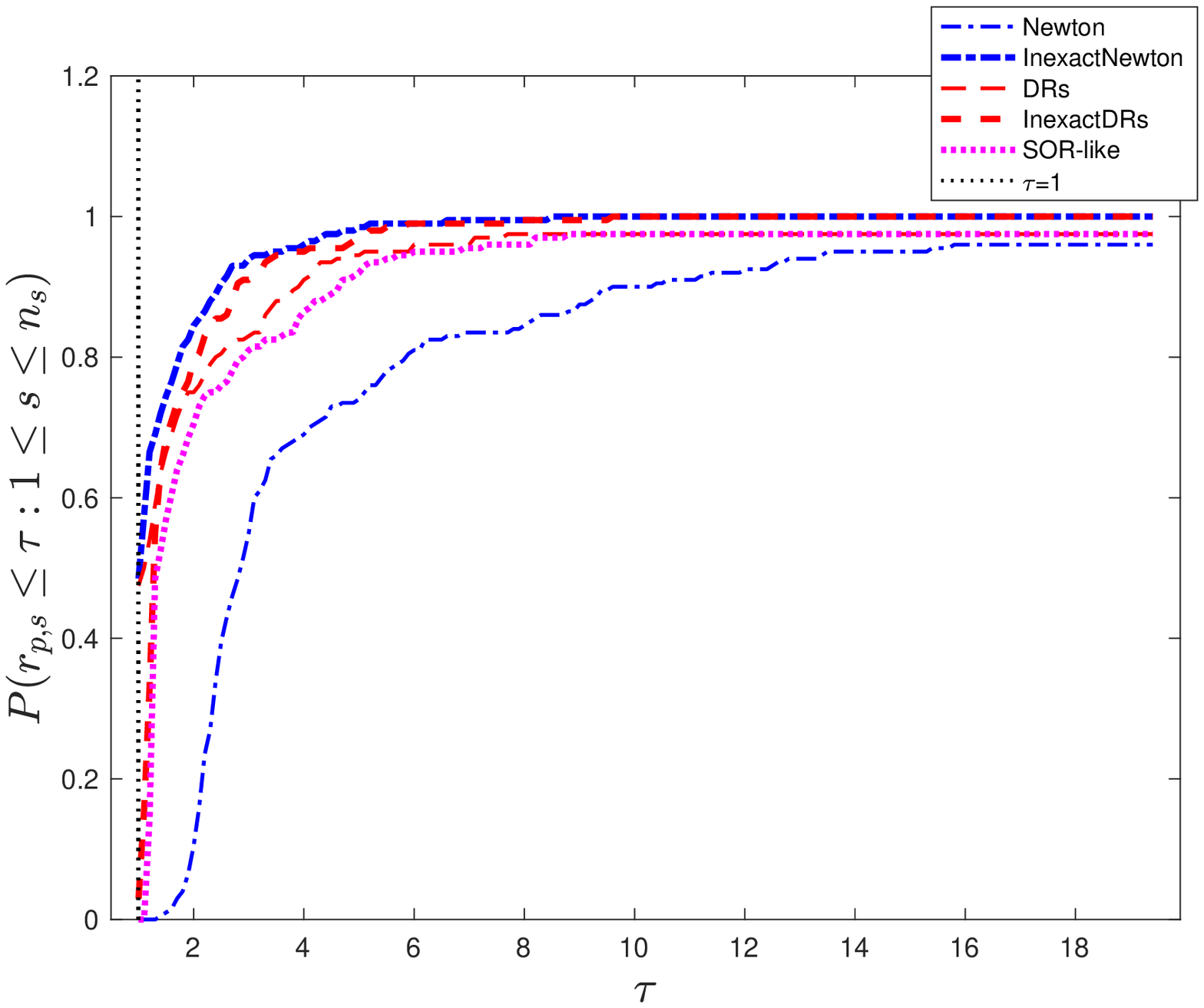}}
\vspace{2ex}\\
\hspace{-0.3 cm}
\resizebox*{0.50\textwidth}{0.26\textheight}{\includegraphics{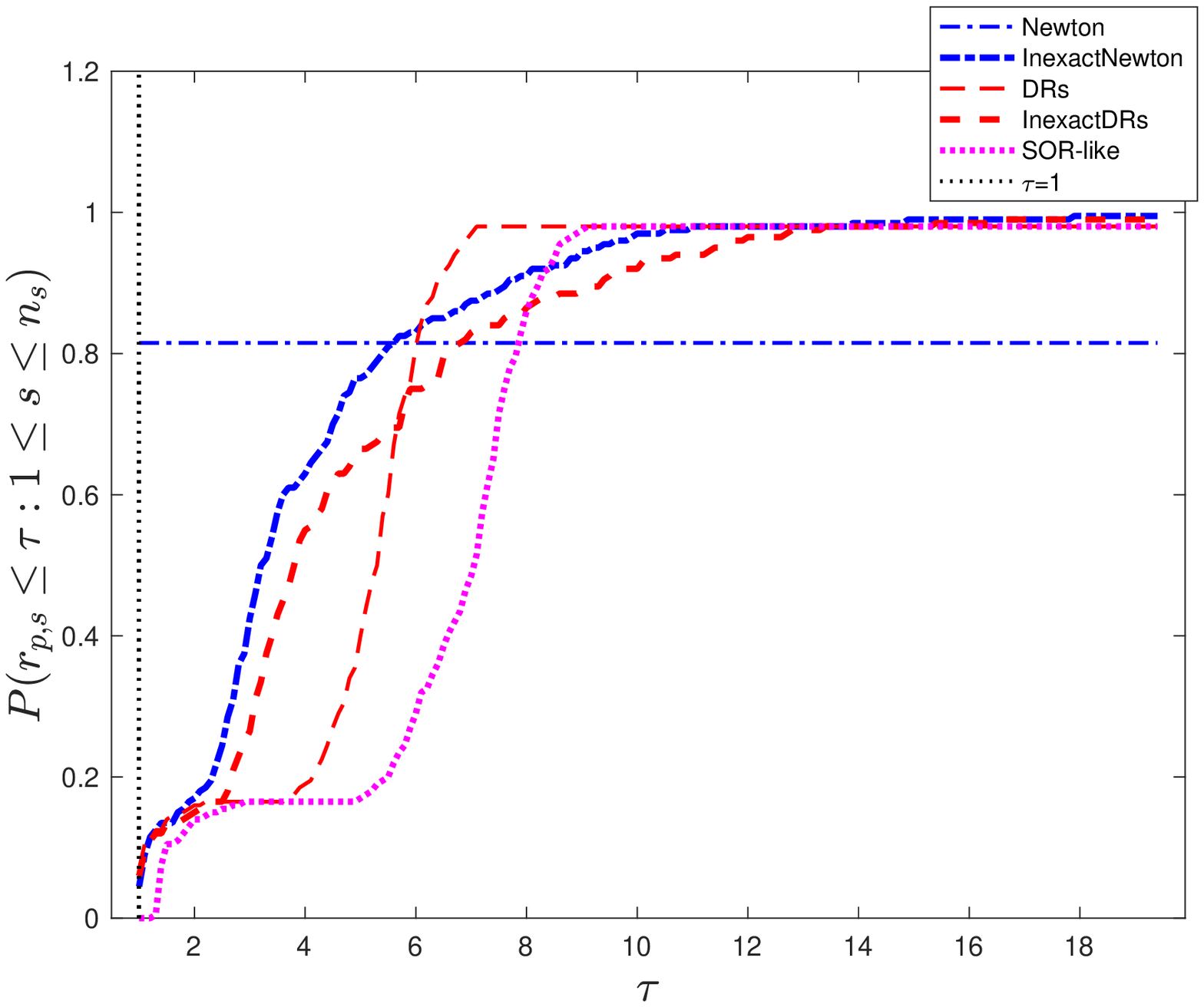}}
& & \hspace{-0.9 cm}
\resizebox*{0.50\textwidth}{0.26\textheight}{\includegraphics{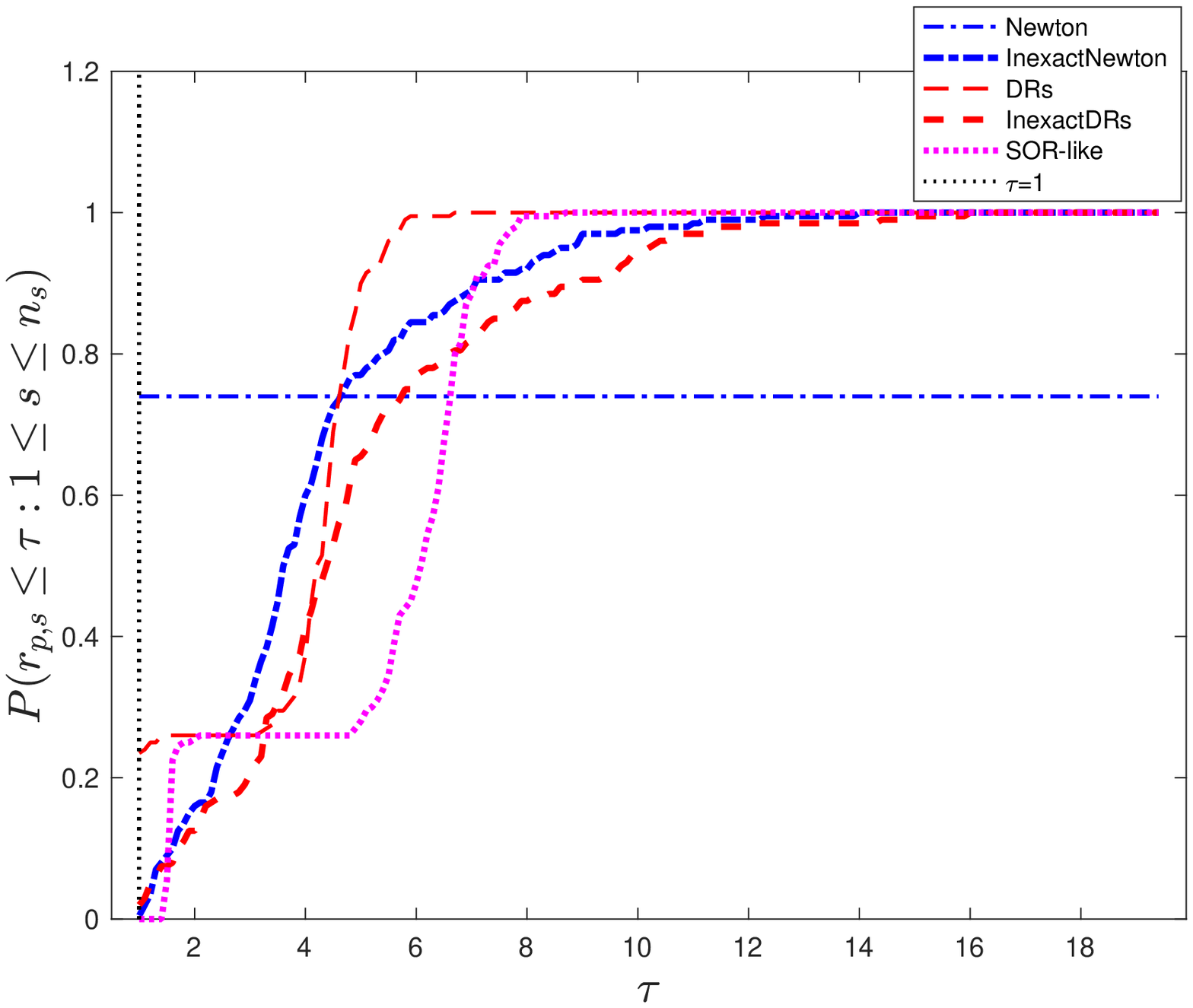}}
\end{tabular}\par
}\vspace{-0.15 cm}
\caption{Performance profile graphics for Example \ref{ex:exam1} with $n = 10000$ (top two plots), $n=3000$ (middle two plots: middle-left, $\text{density} \approx 0.1$; middle-right, $\text{density} \approx 0.4$) and $n = 1500$ (bottom two plots: bottom-left, $\text{density} \approx 0.8$; bottom-right, $\text{density} \approx 1$).}
\label{fig:fig-for-exam1}
\end{figure}

\begin{exam}\label{ex:exam2}
Consider the AVEs \eqref{eq:ave} with $n=10000$, $n = 3000$ and $n = 1500$, respectively, and for each $n$ we generate $200$ AVEs with $\|A^{-1}\|<1$ and thus the Newton and InexactNewton methods are not tested in this example. When $n = 10000$, the mean density of $A$ approximately equal to $0.004$ and the average condition number of $A$ is approximately $311.4557$ $($the smallest and largest value are $48.0533$ and $1901.5$, respectively$)$. When $n = 3000$, two sets of $200$ AVEs are generated, the density of the matrices $A$ in the first set is approximately equal to $0.1$ while it is $0.4$ for the second one. In addition, the average condition number of $A$ is approximately $28.8577$ $($the smallest and largest value is $10.3220$ and $105.3327$, respectively$)$ for the first set and $30.8206$ $($the smallest and largest value is $10.8019$ and $116.8940$, respectively$)$ for the second set. When $n = 1500$, two sets of $200$ AVEs are generated, the density of the matrices $A$ in the first set is approximately equal to $0.8$ while it is approximately equal to $1$ for the second one. In addition, the average condition number of $A$ is approximately $23.5348$ $($the smallest and largest value is $8.3057$ and $90.8945$, respectively$)$ for the first set and $16.6792$ $($the smallest and largest value is $5.4160$ and $68.0478$, respectively$)$ for the second set.
\end{exam}

Numerical results are reported in Table~\ref{table-exam2} and Figure~\ref{fig:fig-for-exam2}. From Table~\ref{table-exam2}, we can conclude that the InexactDRs method has the most wins for solving the tested problems with $n = 10000$ and $n = 3000$; and that the probability that the InexactDRs method is the winner  is about $0.77$, $0.835$ and $0.635$, respectively, for $n=10000$, $n=3000$ with $\text{density} \approx 0.1$ and $n=3000$ with $\text{density} \approx 0.4$. In addition, when $n=1500$, the DRs method has the most wins for solving the tested problems and that the probability that the DRs method is the winner  is about $0.575$ and $0.6$, respectively, for $\text{density} \approx 0.8$ and $\text{density} \approx 1$. Moreover, the DRs method is superior to the SOR-like method in terms of CPU time for this example. Finally, the robustness of all methods is considerably high for this example (the range is from $95.5\%$ to $100\%$). Figure~\ref{fig:fig-for-exam2} clearly shows the performance for all methods. In particular, when $n = 10000$, the DRs method and the SOR-like method become more competitive if we extend our $\tau$ of interest to $6$ (more clearly, see the top-right of Figure~\ref{fig:fig-for-exam2}).
\setlength{\tabcolsep}{5.0pt}
\begin{table}[!h]\small
\centering
\caption{Numerical results for Example~\ref{ex:exam2}.}\label{table-exam2}
\begin{tabular}{|l|l|lllll|}\hline
 \multirow{3}{*}{Method}  &  & \multicolumn{5}{|c|}{$n$}  \\ \cline{3-7}
     &   & \multicolumn{1}{l}{$10000$} & \multicolumn{2}{l}{$3000$} & \multicolumn{2}{l|}{$1500$} \\\cline{2-3}\cline{4-5}\cline{6-7}
     &   Density($\%$) &  $0.38$ & $9.51$ & $38.01$  & $76.02$ & $95.01$ \vextra\\ \hline\hline
\multirow{2}{*}{DRs} & Efficiency($\%$)  & $20.5$ & $16$ & $36$ & $\textbf{57.5}$ & $\textbf{60}$\\
                 & Robustness($\%$) & $\textbf{97.5}$ & $99.5$ & $99$ & $98.5$ & $100$\\\hline
\multirow{2}{*}{InexactDRs} & Efficiency($\%$)  & $\textbf{77}$ & $\textbf{83.5}$ &$\textbf{63.5}$ & $41.5$ & $40$\\
                 & Robustness($\%$) & $95.5$  & $99.5$ &  $\textbf{99.5}$ & $\textbf{100}$ & $100$\\\hline
\multirow{2}{*}{SOR-like} & Efficiency($\%$)  & $0$ & $0$ &  $0$ & $0.5$ & $0$\\
                 & Robustness($\%$) & $\textbf{97.5}$ & $99.5$ & $99$ & $99$ & $100$\\\hline
\end{tabular}
\end{table}

\begin{figure}
{\centering
\begin{tabular}{ccc}
\hspace{-0.3 cm}
\resizebox*{0.50\textwidth}{0.26\textheight}{\includegraphics{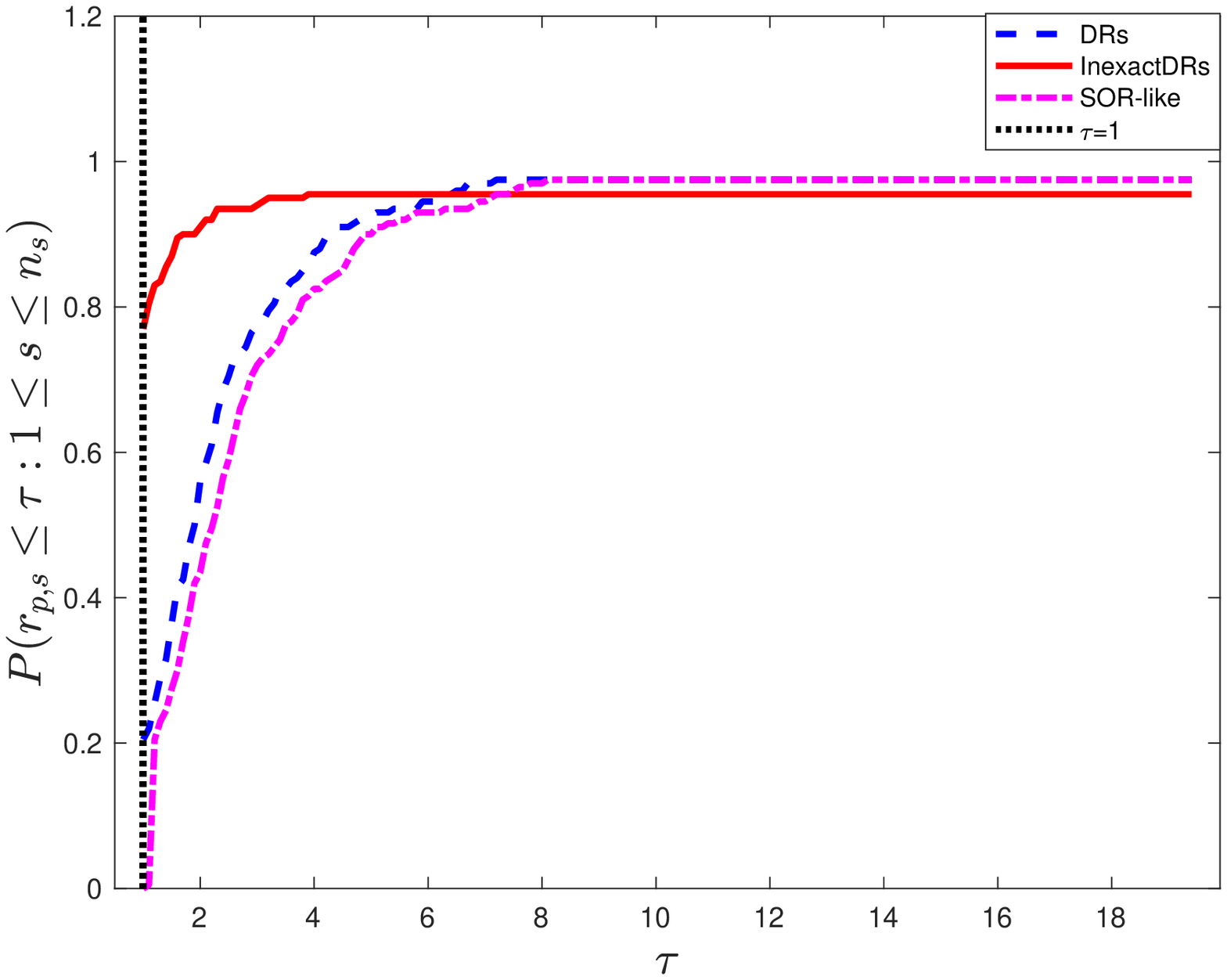}}
& & \hspace{-0.9 cm}
\resizebox*{0.50\textwidth}{0.26\textheight}{\includegraphics{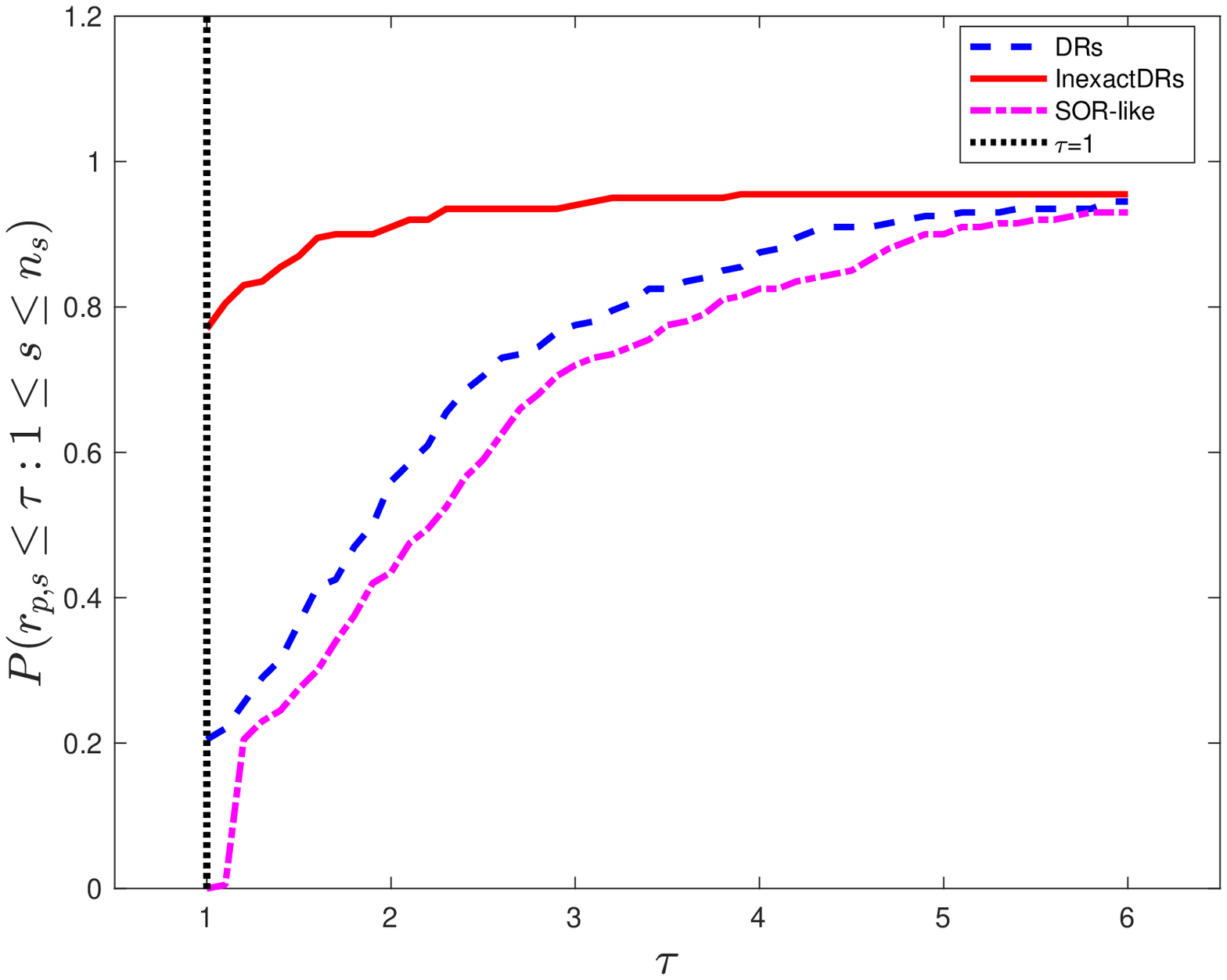}} \vspace{2ex}\\
\hspace{-0.3 cm}
\resizebox*{0.50\textwidth}{0.26\textheight}{\includegraphics{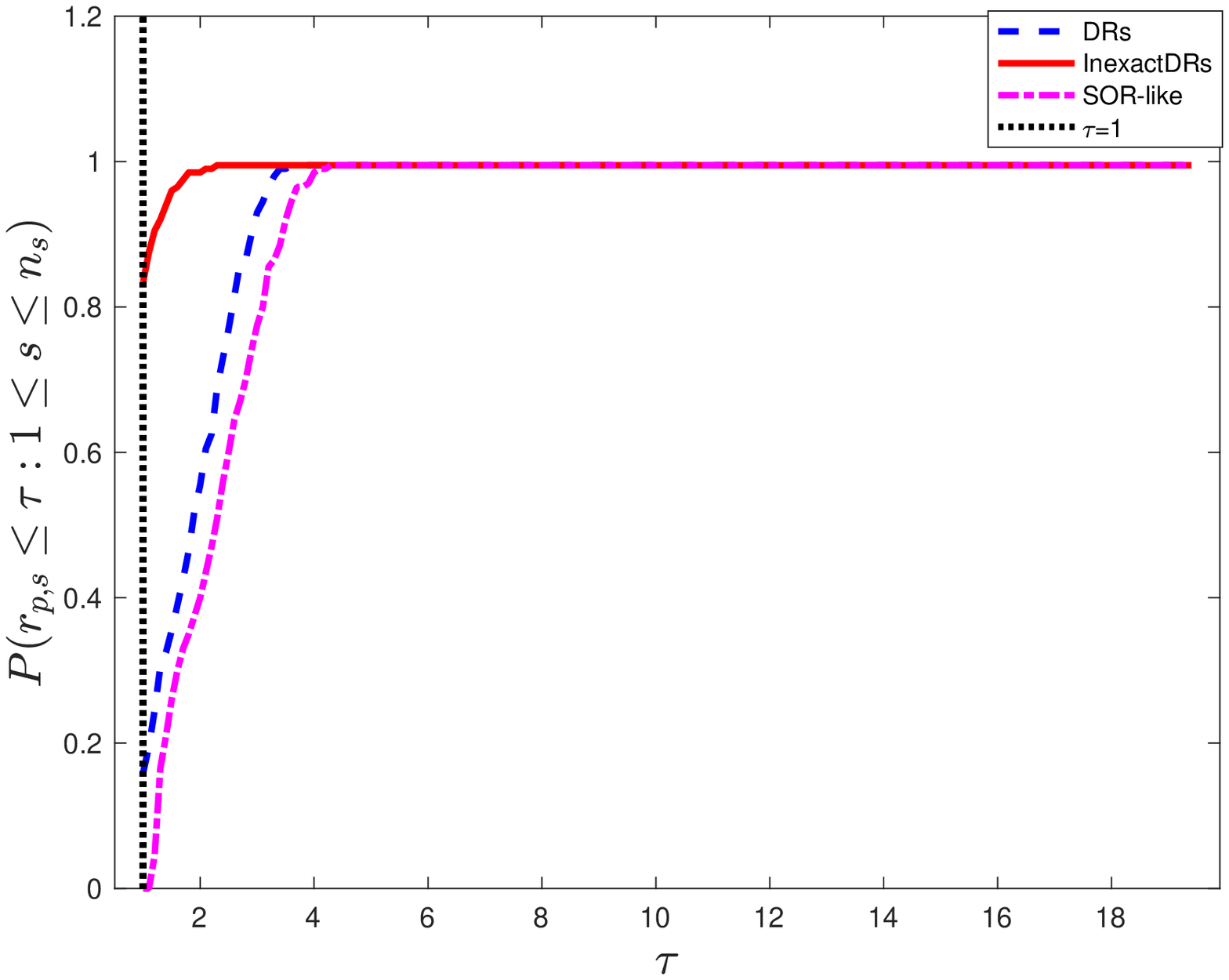}}
& & \hspace{-0.9 cm}
\resizebox*{0.50\textwidth}{0.26\textheight}{\includegraphics{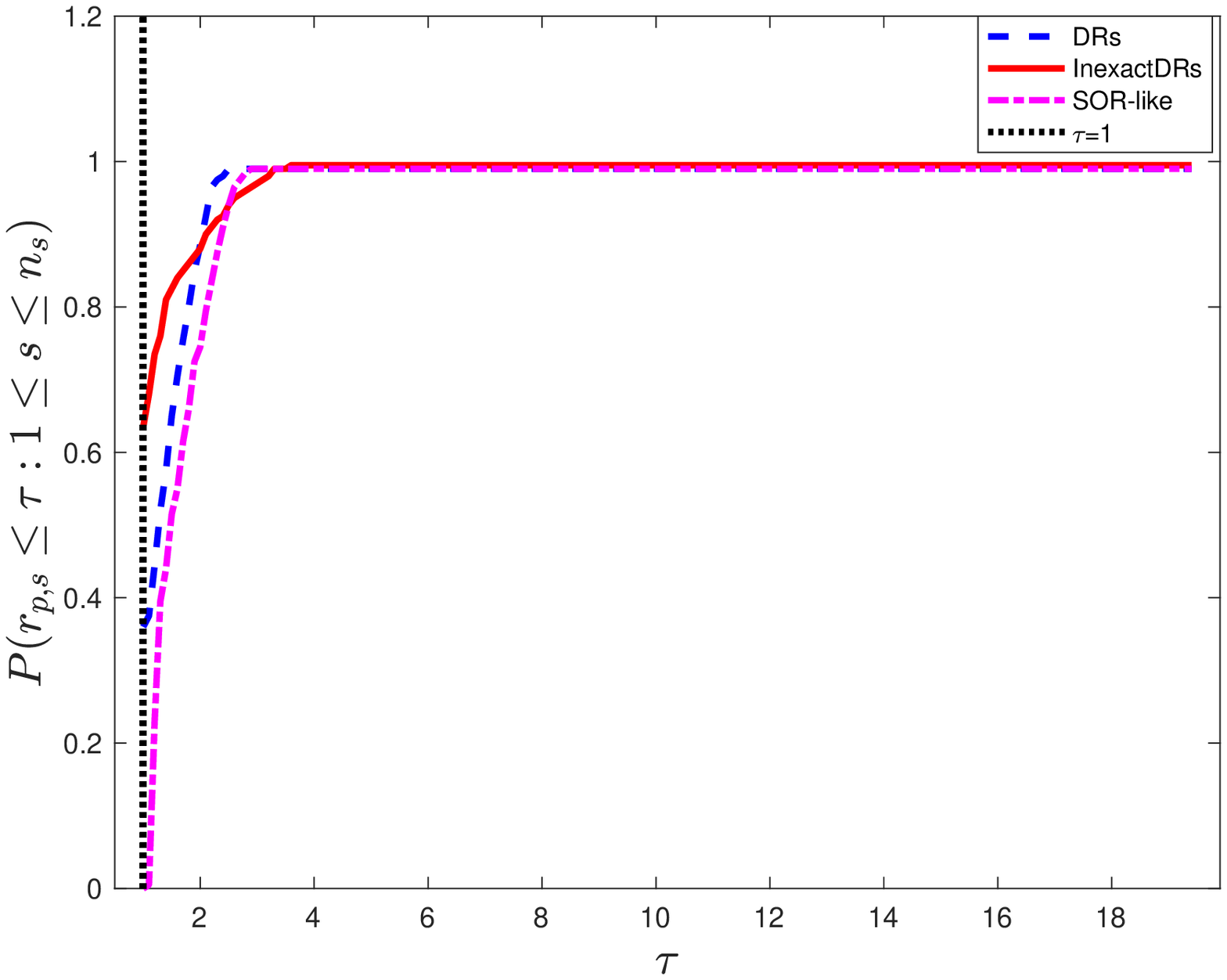}}
\vspace{2ex}\\
\hspace{-0.3 cm}
\resizebox*{0.50\textwidth}{0.26\textheight}{\includegraphics{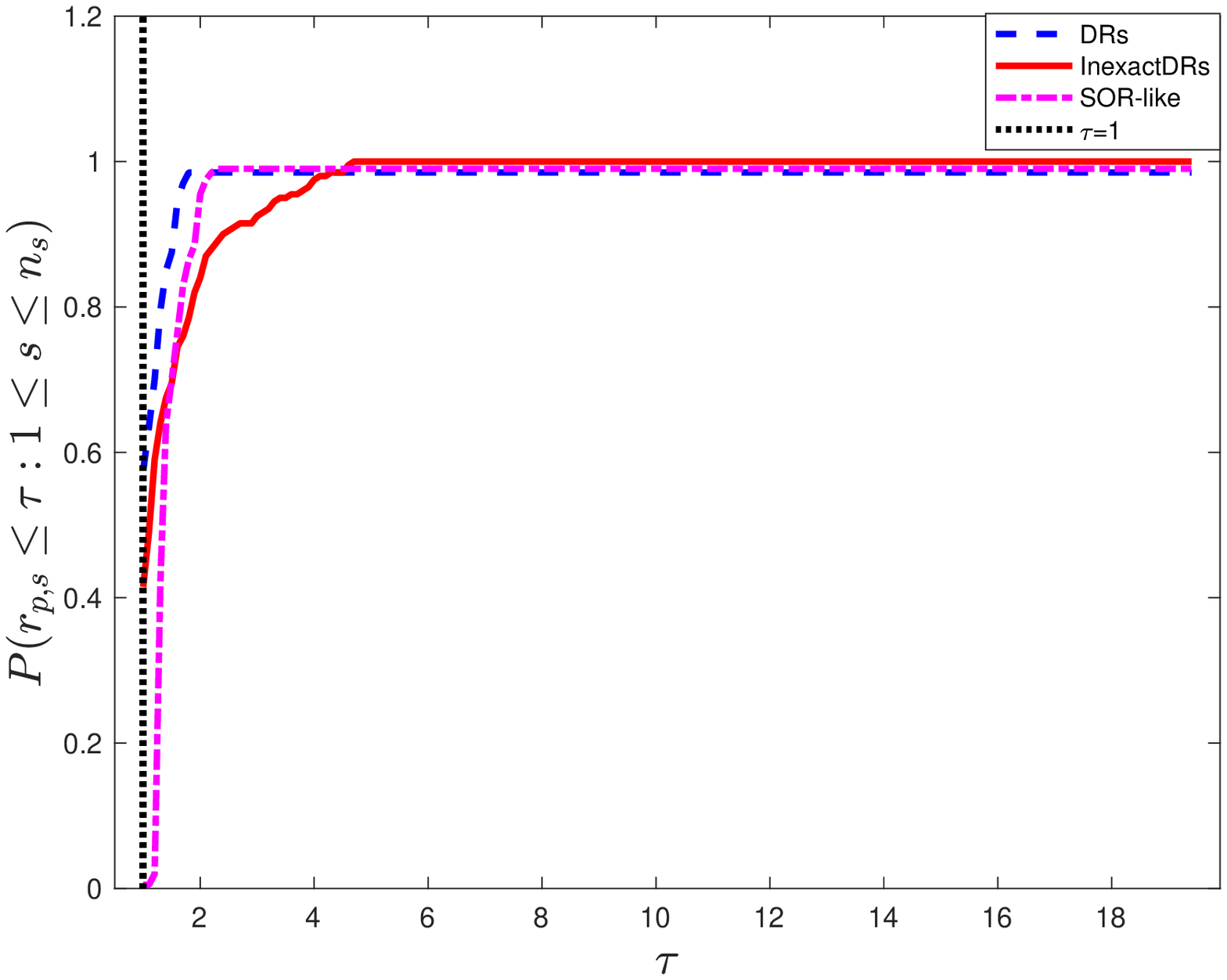}}
& & \hspace{-0.9 cm}
\resizebox*{0.50\textwidth}{0.26\textheight}{\includegraphics{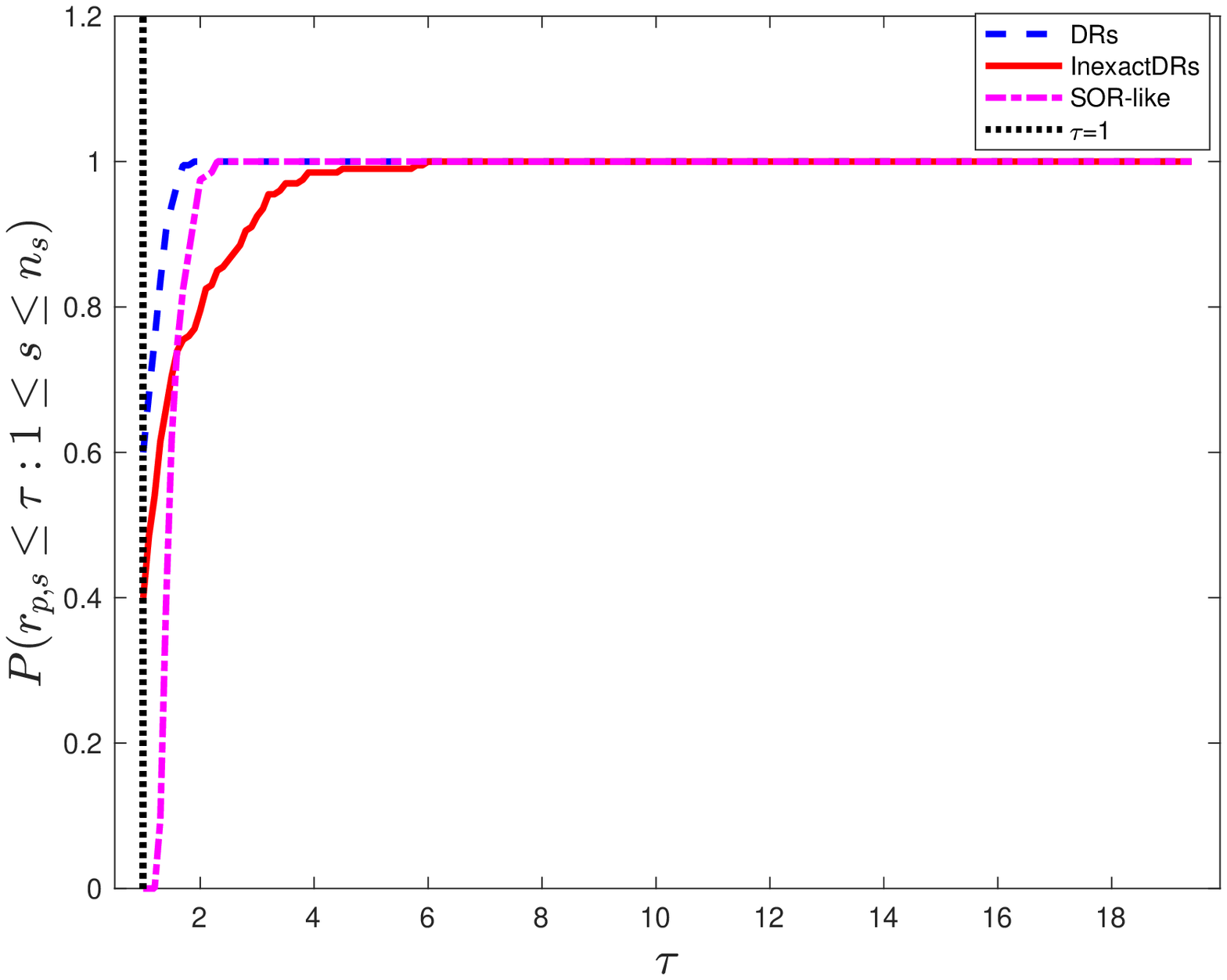}}
\end{tabular}\par
}\vspace{-0.15 cm}
\caption{Performance profile graphics for Example \ref{ex:exam2} with $n = 10000$ (top two plots), $n=3000$ (middle two plots: middle-left, $\text{density} \approx 0.1$; middle-right, $\text{density} \approx 0.4$) and $n = 1500$ (bottom two plots: bottom-left, $\text{density} \approx 0.8$; bottom-right, $\text{density} \approx 1$).}
\label{fig:fig-for-exam2}
\end{figure}

\section{Conclusions}\label{sec:con}
By regarding the absolute value equations (AVE)~\eqref{eq:ave} as a special generalized linear complementarity problem, to solve the AVE~\eqref{eq:ave} is equivalent to finding a zero of the residual function of a nonsmooth projection equation. 
Utilizing the special structure of the induced linear complementarity problem, we designed a  Douglas-Rachford splitting method, and then extended it to allowing approximate solution of the subproblems. The most impressive feature is that we considered the case that $\|A^{-1}\|\le 1$, instead of the existing literature, which can only be applied to the case $\|A^{-1}\|< 1$; in fact, some of them such as the Newton-type methods can only applied to the case with more restrictions. We proved that
\begin{enumerate}[(i)]
\item If the solution set of the AVE~\eqref{eq:ave} is nonempty, then the algorithm, both the exact version and the inexact version, converges globally to a solution, and the rate of convergence is linear;
\item If the solution set is empty, then the generated sequence diverges to infinity, which numerically provides a manner of checking the existence of solutions for the AVE~\eqref{eq:ave} when the sufficient condition of guaranteeing the uniquely solvable is violated.
\end{enumerate}
From the numerical point of view, our methods also have their advantages:
\begin{enumerate}[(i)]
\item Comparing with the Newton-type methods: The coefficient matrix during the iteration is fixed while that of Newton-type methods is varying. This special intrinsic property allows us to deal with the most expensive implementation, the inverse computation, in a `once-for-all' manner. Hence, the cost per iteration reduces to the matrix-vector productions, which is much lower than the inverse computation;
\item Comparing with the SOR-like methods: In SOR-like methods, an auxiliary variable was introduced; hence, the dimension of the system is enlarged to the twice of that the AVE~\eqref{eq:ave}. Although the updating scheme for the variable $x$ is nearly the same as the proposed methods, the updates for the auxiliary variable can not be ignored; this is the case especially where $A$ possesses some structure such that the matrix-vector product in updating $x$ is low cost.
\end{enumerate}

We compared these methods numerically and reported the results, which indicate the proposed methods are competitive.

A byproduct during our analysis is a new sufficient condition for the uniquely solvable of the AVE \eqref{eq:ave}. Although it is { stronger than} the existing one, i.e., $\|A^{-1}\|<1$, it puts some lights, especially the new proof, which is much simple.




\end{document}